\documentclass[12pt,oneside]{amsart}

\usepackage[pdftex]{epsfig}
\newcommand{\Gr}[2]{\psfig{file=#1.pdf,width=#2}}

\usepackage{amssymb}  
\usepackage{latexsym} 
\usepackage{comment}

\usepackage{color}
\usepackage{colonequals}

\definecolor{darkgreen}{rgb}{0,0.5,0}
\usepackage[
        colorlinks, citecolor=darkgreen,
        backref,
        pdfauthor={Ingrid Bauer, Michael Stoll},
        pdftitle={Geometry and arithmetic of primary Burniat surfaces}
]{hyperref}

\usepackage[alphabetic,backrefs,lite]{amsrefs} 

\usepackage[all]{xy} 

\usepackage{enumerate} 

\newcommand{\Z}{{\mathbb Z}}
\newcommand{\Q}{{\mathbb Q}}

\renewcommand{\P}{{\mathbb P}}
\newcommand{\C}{{\mathbb C}}

\newcommand{\BP}{{\mathbb P}}
\newcommand{\BA}{{\mathbb A}}
\newcommand{\calE}{{\mathcal E}}
\newcommand{\calX}{{\mathcal X}}
\newcommand{\To}{\longrightarrow}

\newcommand{\inj}{\hookrightarrow}

\newcommand{\eps}{\varepsilon}

\renewcommand{\div}{\operatorname{div}}
\newcommand{\id}{\operatorname{id}}
\newcommand{\Aut}{\operatorname{Aut}}

\newtheorem{Theorem}{Theorem}[section]
\newtheorem{Lemma}[Theorem]{Lemma}
\newtheorem{Proposition}[Theorem]{Proposition}
\newtheorem{Corollary}[Theorem]{Corollary}
\newtheorem{Conjecture}[Theorem]{Conjecture}

\theoremstyle{definition}
\newtheorem{Definition}[Theorem]{Definition}

\newtheorem{Remark}[Theorem]{Remark}

\numberwithin{equation}{section}

\newcommand{\Abelian}{abelian}

\begin{document}

\title{Geometry and arithmetic of primary Burniat surfaces}

\author{Ingrid Bauer}
\address{Mathematisches Institut,
         Universit\"at Bayreuth,
         95440 Bayreuth, Germany.}
\email{Ingrid.Bauer@uni-bayreuth.de}

\author{Michael Stoll}
\address{Mathematisches Institut,
         Universit\"at Bayreuth,
         95440 Bayreuth, Germany.}
\email{Michael.Stoll@uni-bayreuth.de}
\date{April 12, 2016}
\thanks{This work was done in the framework of the DFG Research Unit 790
        \emph{Classification of algebraic surfaces and compact complex manifolds}}

\begin{abstract}
  We study the geometry and arithmetic of so-called \emph{primary Burniat surfaces},
  a family of surfaces of general type arising as smooth bidouble covers
  of a del Pezzo surface of degree~6 and at the same time as \'etale
  quotients of certain hypersurfaces in a product of three elliptic curves.
  We give a new explicit description of their moduli space and determine
  their possible automorphism groups. We also give an explicit description
  of the set of curves of geometric genus~1 on each primary Burniat surface.
  We then describe how one can try to obtain a description of the set
  of rational points on a given primary Burniat surface~$S$ defined over~$\Q$.
  This involves an explicit description of the relevant twists of the
  \'etale covering of~$S$ coming from the second construction mentioned
  above and methods for finding the set of rational points on a given twist.
\end{abstract}


\maketitle


\section{Introduction}

Whereas it is well-known since Faltings' proof~\cite{Faltings1983} of the
Mordell Conjecture that a curve of general type (and so of genus $\ge 2$)
over~$\Q$ can have only finitely many rational points, not much is known
in general for rational points on \emph{surfaces} of general type.
The Bombieri-Lang Conjecture predicts that the set~$S(\Q)$ of rational
points on a surface~$S$ of general type is not Zariski dense; this
is analogous to Faltings' theorem. This conjecture is open in general,
but when $S$ has a finite \'etale covering~$X$ that is a subvariety
of an abelian variety, then its truth follows from a generalization
of Faltings' result and descent theory, which we recall in
Section~\ref{S:Lang}. We can then
ask whether it is possible to obtain an explicit description of~$S(\Q)$.
The purpose of this paper is to initiate an investigation of this question
by considering a special class of surfaces as a case study. The class
of surfaces we look at are the so-called \emph{primary Burniat surfaces}.

Primary Burniat surfaces
are certain surfaces of general type first constructed by
Burniat~\cite{burniat} in the 1960s as smooth bidouble covers of a del Pezzo
surface of degree~$6$.
Their moduli space $M$ forms a~$4$-dimensional irreducible
connected component  of the moduli space of surfaces of general type as shown
by Mendes Lopes and Pardini ~\cite{mp} using an argument related to the degree of the
bicanonical map. This result was reproved by Bauer and Catanese
in~\cite{burniat1}, where the authors also show the rationality of~$M$.
Their proof makes use of the fact (also proved
in the same paper) that Burniat's surfaces can also be obtained as \'etale
quotients of certain hypersurfaces in products of three elliptic curves,
a construction first considered by Inoue~\cite{inoue}. It is this description
that allows us to show that Lang's Conjectures on rational points on
varieties of general type hold for primary Burniat surfaces.
In fact, we recall a more general statement for projective varieties admitting
an \'etale covering that embeds in a product of \Abelian\ varieties and
curves of genus at least two. This applies in particular for certain so-called
Inoue type varieties; see Section~\ref{S:Lang}.

After recalling the two different constructions of primary Burniat surfaces
in Section~\ref{S:Burniat},
we produce an explicit description of their moduli space as a subset
of~$\BA^6$ in Section~\ref{S:moduli}.
We then use this description to classify the possible automorphism
groups of primary Burniat surfaces. Generically, this automorphism group
is $(\Z/2\Z)^2$, but its order can get as large as~$96$; see Section~\ref{S:auto}.

By the general theorem given in Section~\ref{S:Lang}, we know that the
set of rational points on a primary Burniat surface~$S$ consists of the points
lying on the finitely many curves of geometric genus~$\le 1$ on~$S$, together
with finitely many `sporadic' points. To describe this set~$S(\Q)$ explicitly,
we therefore need to determine the set of low-genus curves on~$S$.
This is done in Section~\ref{S:low}. The result is that $S$ does not contain
rational curves, but does always contain six smooth curves of genus~$1$, and may
contain up to six further (singular) curves of geometric genus~$1$.

In the last part of the paper, we discuss how the set $S(\Q)$ can be determined
explicitly for a primary Burniat surface~$S$ defined over~$\Q$.
We restrict to the case that the three genus~$1$ curves whose product contains
the \'etale cover~$X$ of~$S$ are individually defined over~$\Q$. We can then
write down equations for~$X$ of a somewhat more general (but similar) form
than those used in Section~\ref{S:moduli} where we were working over~$\C$.
We explain how to determine the finitely many twists of the covering $X \to S$
that may give rise to rational points on~$S$, see Section~\ref{S:explicit},
and we discuss what practical methods are available for the determination
of the sets of rational points on these twists, see Section~\ref{S:find}.
We conclude with a number of examples.


\section{Rational points on varieties of general type} \label{S:Lang}

We start giving a short account on some conjectures, which were formulated
by S.~Lang\footnote{Sometimes these conjectures are referred
to as `Bombieri-Lang Conjectures', but Enrico Bombieri insists that he only
ever made these conjectures for surfaces.}~\cite{langhyp}
after discussions with an input from several other mathematicians.

\begin{Conjecture}[Weak Lang Conjecture]\label{Conj:SBomLa1}
  If $X$ is a smooth projective variety of general type defined over a number field~$K$,
  then the set~$X(K)$ of $K$-rational points on~$X$ is not Zariski-dense in~$X$.
\end{Conjecture}

For $X$ as in this conjecture, one defines the \emph{exceptional set}~$N$
of~$X$ as the union of all the images of non-constant morphisms $f \colon A \to X$
from an \Abelian\ variety~$A$ to~$X$, defined over~$\bar{K}$.

\begin{Conjecture}[Strong Lang Conjecture]\label{Conj:SBomLa2}
  Let $X$ be a smooth projective variety of general type defined over a number field~$K$.
  The exceptional set~$N$ of~$X$ is a proper closed subvariety of~$X$,
  and for each finite field extension $K \subset L$, the set $X(L) \setminus N(L)$
  is finite.
\end{Conjecture}

In general these conjectures are wide open. They hold
for curves (by Faltings' proof of the Mordell Conjecture), and more generally
for subvarieties of \Abelian\ varieties, see~\cite{Faltings1994}.
However, even in
these cases it is far from clear that a suitable explicit description of~$X(K)$
can actually
be determined. In the case of curves, there has been considerable progress
in recent years (see for example~\cite{StollFiniteDescent}),
indicating that an effective (algorithmic) solution might be possible.

It is therefore natural
to consider the case of surfaces next, where so far only very little is known.

Let~$S$ be a smooth projective surface of general type over a number field~$K$.
In this case, Conjectures \ref{Conj:SBomLa1} and~\ref{Conj:SBomLa2} state
that there are only finitely
many curves of geometric genus 0 or~1 on~$S$
(possibly defined over larger fields)
and that the set of $K$-rational points on~$S$ outside these curves is finite.
We call the curves of geometric genus at most~$1$ on~$S$ the \emph{low-genus curves}
on~$S$, and we call the $K$-rational points outside the low-genus curves
\emph{sporadic $K$-rational points} on~$S$.

An explicit description of $S(K)$ would then consist of the finite list
of the low-genus curves on~$S$, together with the finite list
of sporadic $K$-rational points.
Given a class of surfaces of general type, it is therefore natural to proceed in the following way:
\begin{enumerate}[(1)]
  \item \label{I.PFC}
        prove the finiteness of the set of low-genus curves on surfaces in the class;
  \item \label{I.AlgC}
        develop methods for determining this set explicitly;
  \item \label{I.PFP}
        prove the finiteness of the set of sporadic $K$-rational points on surfaces in the class;
  \item \label{I.AlgP}
        develop methods for determining the set of sporadic points explicitly.
\end{enumerate}
Regarding point~\eqref{I.PFC}, the statement was proved by Bogomolov in the
1970s~\cite{bogomolov} under the condition that the following inequality is satisfied:
\[ c_1^2 (S) = K^2_S > c_2 (S), \quad \text{or equivalently,} \quad K^2_S > 6 \chi (S). \]
Under the same assumption, Miyaoka~\cite{miyaoka} gave an upper bound for
the canonical degree $K_S \cdot C$ of a curve~$C$ of geometric genus $\le 1$
on~$S$. A result by S.~Lu~\cite{lu} gives a similar conclusion when
the irregularity~$q(S)$ is at least~$2$.
For another approach that applies for example to diagonal varieties in
arbitrary dimension, see~\cites{BrownawellMasser,Voloch}.

Results of Faltings~\cite{Faltings1994} and Kawamata~\cite{kawamata}*{Theorem~4}
imply that Conjecture~\ref{Conj:SBomLa2} holds for subvarieties
(of general type) of \Abelian\ varieties.
Combining this with a generalization (see for example~\cite{serre}*{\S4.2})
of a result due to Chevalley and Weil~\cite{ChevalleyWeil}, this gives the following
(see also~\cite{HindrySilverman}*{Prop.~F.5.2.5~(ii)} for the `weak' form).

\begin{Theorem} \label{T:general}
  Let $X$ be a smooth projective variety of general type
  over a number field~$K$ which admits a finite \'etale
  covering $\pi \colon \hat{X} \to X$ such that $\hat{X}$ is contained in a product~$Z$
  of \Abelian\ varieties and curves of higher genus (i.e., genus $\geq 2$).
  \begin{enumerate}[\upshape (1)]
    \item $X(K)$ is not Zariski dense in $X$.
    \item If $\hat{X}$ does not contain any translate of a positive dimensional \Abelian\ subvariety
          of the \Abelian\ part of~$Z$, then $X(K)$ is finite.
    \item $X$ satisfies Conjecture~\ref{Conj:SBomLa2}.
  \end{enumerate}
\end{Theorem}

\begin{Remark} \label{R:twists}
  For the purpose of actually determining~$X(K)$, we assume in addition that
  the covering~$\pi$ is geometrically Galois, so that $\hat{X}$ is an $X$-torsor
  under a finite $K$-group scheme~$G$ that we also assume to act on~$Z$. ($G$ is allowed
  to have fixed points on~$Z$, but not on~$\hat{X}$.) This allows us to use
  the machinery of descent (see for example~\cite{SkorobogatovBook}).
  There is then a finite collection
  of twists $\pi_\xi \colon \hat{X}_\xi \to X$ of the torsor $\hat{X} \to X$ such that
  $X(K) = \bigcup_\xi \pi_\xi\bigl(\hat{X}_\xi(K)\bigr)$.
  This collection of twists can be taken to be those that are unramified outside
  the set of places of~$K$ consisting of the archimedean places, the places
  dividing the order of~$G$ and the places of bad reduction for~$X$ or~$\pi$.

  Since $G$ also acts on~$Z$ by assumption,
  we have $\hat{X}_\xi \subset Z_\xi$ for the corresponding twists of~$Z$.
  The general form of such a twist is a product of a principal homogenous space
  for some \Abelian\ variety with some curves of higher genus or Weil restrictions of such
  curves over some finite extension of~$K$ (the latter can occur when the action
  of~$G$ on~$Z$ permutes some of the curve factors).
  Write $Z_\xi = Z_\xi^{(1)} \times Z_\xi^{(2)}$, where $Z_\xi^{(1)}$ is the `\Abelian\ part'
  and $Z_\xi^{(2)}$ is the `curve part'.
  By Faltings' theorem, $Z_\xi^{(2)}$ has only finitely many $K$-rational points.
  Fixing such a point~$P$, the fiber~$\hat{X}_{\xi,P}$ above~$P$
  of $\hat{X}_\xi \to Z_\xi^{(2)}$ is a subvariety of the principal homogeneous space~$Z_\xi^{(1)}$.
  So either $\hat{X}_{\xi,P}(K)$ is empty, or else we can consider $Z_\xi^{(1)}$ as an \Abelian\
  variety, so that Faltings' theorem applies to~$\hat{X}_{\xi,P}$.
  This shows that $\hat{X}_{\xi,P}(K)$ is contained in a finite union of translates
  of \Abelian\ subvarieties of~$Z_\xi^{(1)}$ that are contained in~$\hat{X}_{\xi,P}$.
  Since there are only finitely many~$P$, the corresponding statement holds for~$\hat{X}_\xi$
  as well.

  The advantage of this variant is that it allows us to work with $K$-rational points
  on~$\hat{X}$ and its twists, rather than with $K'$-rational points on~$\hat{X}$
  for a (usually rather large) extension~$K'$ of~$K$. So if we have a way
  of determining the translates of \Abelian\ subvarieties of positive dimension
  contained in~$\hat{X}$ and also of finding the finitely many $K$-rational points
  outside these translates for each twist of~$\hat{X}$, we can determine~$X(K)$.
  We will pursue this approach when $X$ is a primary Burniat surface in
  Sections \ref{S:explicit}, \ref{S:find} and~\ref{S:examples} below.
\end{Remark}

Theorem~\ref{T:general} applies to a generalization of Inoue's construction of Burniat's surfaces,
the so-called \emph{classical Inoue type varieties}, which were introduced in~\cite{bacainoue},
and which proved to be rather useful in the study of deformations in the large of surfaces
of general type.
In fact, it turned out that for Inoue type varieties
one can show under certain additional conditions that any smooth variety
homotopically equivalent to an Inoue type variety is in fact
an Inoue type variety. This often allows to determine and describe
the connected components of certain moduli spaces of surfaces of general type, since it implies
that a family yields a closed subset in the moduli space.
For details we refer to the original article~\cite{bacainoue}.
A classical Inoue type variety is of general type, and if it is defined over a number field~$K$,
then it satisfies the assumptions
of Theorem~\ref{T:general} and also Remark~\ref{R:twists}.
Taking $K = \Q$ (to keep the computations feasible),
we obtain a large class of examples of varieties of general type
for which one can try to use the approach of Remark~\ref{R:twists} to determine the
set of rational points explicitly.


\section{Primary Burniat Surfaces} \label{S:Burniat}

The so-called \emph{Burniat surfaces} are several families of surfaces of general type
with $p_g=0$, $K^2 =6,5,4,3,2$, first constructed by P.~Burniat in~\cite{burniat}
as singular \emph{bidouble covers} (i.e., $(\Z/2 \Z)^2$ Galois covers) of the projective
plane~$\P^2$ branched on a certain configuration of nine lines.
Later, M.~Inoue~\cite{inoue} gave another construction of surfaces `closely related to Burniat's surfaces'
with a different technique as $(\Z/2 \Z)^3$-quotients of an invariant hypersurface
of multi-degree~$(2,2,2)$ in a product of three elliptic curves.
In~\cite{burniat1} it is shown that these two constructions give exactly the same surfaces.

In fact, concerning statements about rational points on Burniat surfaces,
Inoue's construction is much more useful than the original construction of Burniat,
as will be clear soon.

We briefly recall Burniat's original construction of Burniat surfaces with $K^2=6$,
which were called \emph{primary} Burniat surfaces in~\cite{burniat1}.

Consider three non-collinear points $p_1, p_2, p_3 \in \P^2$,
which we take to be the coordinate points, and denote
by $Y \colonequals \hat{\P}^2(p_1,p_2,p_3)$ the blow-up of~$\P^2$ in~$p_1,p_2,p_3$.
Then $Y$ is a del~Pezzo surface of degree~$6$
and it is the closure of the graph of the rational map
$$\epsilon \colon \P^2 \dashrightarrow \P^1 \times \P^1 \times \P^1$$
such that
$$\epsilon (y_1 : y_2: y_3) = \bigl((y_2 : y_3), (y_3 : y_1), (y_1 : y_2)\bigr).$$

We denote by $e_i$ the exceptional curve
lying over $p_i$ and by $D_{i,1}$ the unique effective divisor in $|l - e_i- e_{i+1}|$,
i.e., the proper transform of the line $y_{i-1} = 0$, side of the
triangle joining the points $p_i, p_{i+1}$. Here $l$ denotes the total transform of a general line in $\P^2$.

Consider on $Y$ the divisors
$$ D_i = D_{i,1} + D_{i,2} + D_{i,3} + e_{i+2} \in |3l - 3e_i - e_{i+1}+e_{i+2}|,$$
where $D_{i,j} \in |l - e_i|$, for $j = 2,3$, $D_{i,j} \neq D_{i,1}$,
is the proper transform of another line through~$p_i$ and $D_{i,1} \in |l - e_i- e_{i+1}|$
is as above. Assume also that all the corresponding lines in~$\P^2$ are distinct,
so that $D \colonequals \sum_i D_i$ is a reduced divisor.
In this description, the subscripts are understood as residue classes modulo~3.

If we define the divisor $L_i \colonequals 3l - 2 e_{i-1} - e_{i+1}$, then
$$D_{i-1} + D_{i+1} \equiv 6l - 4 e_{i-1} - 2e_{i+1} = 2 L_i,$$
and we can consider the associated bidouble cover $S \rightarrow Y$
branched on~$D$.

From the birational point of view, the function field of $S$ is obtained from the function field of $\mathbb P^2$ by adjoining
$$
\sqrt{\frac{\Delta_1}{\Delta_3}} \ , \ \sqrt{\frac{\Delta_2}{\Delta_3}} \,,
$$
where $\Delta_i \in \mathbb{C} [x_0,x_1,x_2]$ is the cubic polynomial whose zero set has strict transform~$D_i$.

A biregular model is given by the following: let $D_i= \div(\delta_i)$ and let $u_i$ be a fibre coordinate of the geometric line bundle $\mathbb L_i$ whose sheaf of holomorphic sections is $\mathcal O_{Y}(L_i)$. Then $S \subset  \mathbb L_1 \oplus \mathbb L_2 \oplus \mathbb L_3$ is given by the equations:
\begin{align*}
   u_1u_2 = \delta_1 u_3  \ , & \ \ u_1^2 = \delta_3 \delta_1 ,  \\
      u_2u_3 = \delta_2 u_3  \ , & \ \ u_2^2 = \delta_1 \delta_2 , \\
   u_3u_1 = \delta_3 u_2   \ , & \ \ u_3^2 = \delta_2 \delta_3  \,.
\end{align*}
This shows clearly that $S \rightarrow Y$ is a Galois cover with group $(\mathbb Z / 2 \mathbb Z)^2$ and we see the intermediate double covers as the double covers of $Y$ branched on $D_i + D_j$.

Compare Figure~\ref{Figburniat}.

\begin{figure}[htb]
  \hrulefill
  \vspace{5mm}
  \begin{center}
    \Gr{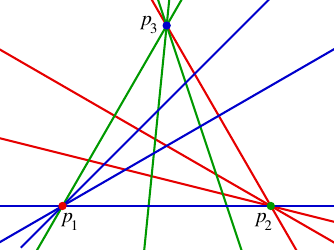}{0.6\textwidth}
  \end{center}
  \caption{The configuration of the branch locus in Burniat's construction.
           $D_1$ is blue, $D_2$ is red and $D_3$ is green.}
  \label{Figburniat}
  \hrulefill
\end{figure}

Observe that $S$ is nonsingular exactly when the divisor $D$
does not have points of multiplicity~3 (there cannot be points of higher multiplicities).

\begin{Definition}
  A \emph{primary Burniat surface} is a surface~$S$ as constructed above that is non-singular.
\end{Definition}

\begin{Remark}
  A primary Burniat surface $S$ is a minimal surface with $K_S$ ample.
  It has invariants $K_S^2 =6$ and $p_g(S) = q(S) =0$.
\end{Remark}

We now describe Inoue's construction in detail, since we shall make extensive use of it.
We pick three elliptic
curves $E_1$, $E_2$, $E_3$, together with points $T_j \in E_j$ of order~2.
On each~$E_j$, we then have two commuting involutions, namely the translation
by~$T_j$ and the negation map. We now consider the group
$\Gamma = \langle \gamma_1, \gamma_2, \gamma_3 \rangle$
of automorphisms of the product $E_1 \times E_2 \times E_3$ with
\begin{align*}
   \gamma_1 &\colon (P_1, P_2, P_3) \longmapsto (P_1, P_2+T_2, -P_3+T_3) \\
   \gamma_2 &\colon (P_1, P_2, P_3) \longmapsto (-P_1+T_1, P_2, P_3+T_3) \\
   \gamma_3 &\colon (P_1, P_2, P_3) \longmapsto (P_1+T_1, -P_2+T_2, P_3) \,.
\end{align*}
It can be checked that none of the $\gamma_j$ and none of the
$\gamma_i \gamma_j$ (with $i \neq j$)
have fixed points on $E_1 \times E_2 \times E_3$. However, the automorphism
$\gamma = \gamma_1 \gamma_2 \gamma_3$
is given by $(P_1, P_2, P_3) \mapsto (-P_1, -P_2, -P_3)$, so it has $4^3$
fixed points, namely the points $(P_1, P_2, P_3)$ such that $P_j$ is a
2-torsion point on~$E_j$ for each~$j$.

Now let $x_j$ be a function of degree~2 on~$E_j$ that is invariant under the negation
map and changes sign under translation by~$T_j$. Then the product $x_1 x_2 x_3$
is invariant under~$\Gamma$, and so the surface $X \subset E_1 \times E_2 \times E_3$
defined by $x_1 x_2 x_3 = c$ is also $\Gamma$-invariant.

\begin{Remark}
  It can be checked that $X$ is smooth unless it contains a fixed point of~$\gamma$
  or $c = 0$~\cite{inoue}. So as long
  as we avoid a finite set of values for~$c$, the surface~$X$ will be smooth
  and $\Gamma$ acts on it without fixed points. We let $S = X/\Gamma$ be the
  quotient. Then $\pi \colon X \to S$ is an unramified covering that is Galois
  with Galois group~$\Gamma$.
  Both $X$ and~$S$ are surfaces of general type.
\end{Remark}

In \cite{burniat1}*{Theorem 2.3} the following is shown:
\begin{Theorem}
  Primary Burniat surfaces are exactly the surfaces $S = X/\Gamma$ as above
  such that $\Gamma$ acts freely on $X$.
\end{Theorem}

\begin{Remark}
  This implies that primary Burniat surfaces are classical Inoue type varieties.
\end{Remark}


\section{An explicit moduli space for primary Burniat surfaces} \label{S:moduli}

In~\cite{burniat1}*{Theorem 4.1, Theorem 4.2} the following result is shown:

\begin{Theorem} \strut
\begin{enumerate}[\upshape(1)]
\item Let $S$ be a smooth complex projective surface that is homotopically
      equivalent to a primary Burniat surface. Then $S$ is a primary Burniat
      surface.
\item The subset $M$ of the Gieseker moduli space corresponding to primary
      Burniat surfaces is an irreducible connected component, normal and
      rational of dimension~$4$.
\end{enumerate}
\end{Theorem}
In particular, this shows that if $[S] \in M$, then $S$ is in fact a primary Burniat surface,
i.e., it can be obtained via Burniat's and Inoue's constructions.

Still, in~\cite{burniat1} no explicit model of~$M$ is given.
The goal of this section is to provide such an explicit moduli space
based on Inoue's construction.
As a by-product, we obtain another proof of its rationality.
We will also use our approach to classify the possible automorphism groups of
primary Burniat surfaces, see Section~\ref{S:auto} below.
We work over~$\C$.

We first consider a curve~$E$ of genus~$1$ given as a double
cover $\pi_E \colon E \to \P^1$ ramified in four points,
together with a point $T \in J$ of order~$2$, where $J$ is the Jacobian
elliptic curve of~$E$ (so that $E$ is a principal homogeneous space for~$J$;
more precisely, $E$ is a $2$-covering of~$J$ via $P \mapsto P - \iota_E(P)$,
where $\iota_E$ is the involution on~$E$ induced by~$\pi_E$).
There is then a coordinate~$x$ on~$\P^1$, unique up to scaling
and replacing $x$ by~$1/x$,
such that $x(\pi_E(P+T)) = -x(\pi_E(P))$ for all $P \in E$. (To see the uniqueness,
fix a point $Q \in E$ with $Q - \iota_E(Q) = T$. Then the divisor of~$x \circ \pi_E$ must be
$\pm\bigl((Q) + (\iota_E(Q)) - (Q+T') - (\iota_E(Q)+T')\bigr)$, where $T' \neq T$ is another point
of order~$2$.) As a double cover of~$\P^1$ via~$x \circ \pi_E$, $E$ is then given
by an equation of the form $y^2 = x^4 + a x^2 + b$. We fix the scaling of~$x$
up to a fourth root of unity by requiring $b = 1$. This fixes $a$ up to a sign.
So the moduli space of pairs $(E \to \P^1, T)$ as above
(i.e., $E \to \P^1$ is a double cover and $T$ is a point of order~$2$
on the Jacobian of~$E$) is $\BA^1 \setminus \{4\}$,
with the parameter for $y^2 = x^4 + a x^2 + 1$ being~$a^2$.
If we identify $E$ with~$J$ by taking the origin to be one of the ramification
points of~$x$, then $y$ changes sign under negation.
In general, the fibers $\{P,P'\}$ of $\pi_E \colon J \cong E \to \P^1$ are characterized
by $P + P' = Q$ for some point $Q \in J$; then the automorphism~$\iota_E$ of~$E$
corresponding to changing the sign of~$y$ is given by $P \mapsto Q - P$.

Consider the subvariety $\calE$ given by~$y^2 = x^4 + a x^2 z^2 + z^4$
in~$\P(1,2,1) \times (\BA^1 \setminus \{\pm 2\})$, where $x,y,z$ are the
coordinates of~$\P(1,2,1)$ and $a$ is the coordinate on~$\BA^1$
(then $x/z$ is the function denoted~$x$ in the previous paragraph).
Then by the above, the group of automorphisms of~$\calE$ that
respect the morphism to~$\BP^1$ given by $(x:z)$ and are compatible
with the fibration over~$\BA^1 \setminus \{\pm 2\}$ is generated by
\begin{align*}
  \bigl((x:y:z), a\bigr) &\longmapsto \bigl((ix:y:z), -a\bigr)\,, \\
  \bigl((x:y:z), a\bigr) &\longmapsto \bigl((x:-y:z), a\bigr) \quad\text{and}\\
  \bigl((x:y:z), a\bigr) &\longmapsto \bigl((z:y:x), a\bigr)\,.
\end{align*}
The fibers $(E_a \to \P^1,T_a)$ and~$(E_b \to \P^1,T_b)$ over $a$ and~$b$ are isomorphic
(in the sense that there is an isomorphism $E_a \to E_b$ that induces an automorphism
of~$\P^1$ and identifies $T_a$ with~$T_b$) if and only if $b = \pm a$,
and $(E_a \to \P^1,T_a)$ has extra automorphisms if and only if~$a = 0$.

We now use a similar approach to classify surfaces $X \subset E_1 \times E_2 \times E_3$
as in Inoue's construction of primary Burniat surfaces.
By the above, we can define $X$ as a subvariety of
$E_1 \times E_2 \times E_3 \subset \P(1,2,1)^3$ by equations
\begin{equation} \label{E:eqns} \renewcommand{\arraystretch}{1.2}
  \left\{
  \begin{array}{r@{{}={}}l}
    y_1^2 & x_1^4 + a_1 x_1^2 z_1^2 + z_1^4 \\
    y_2^2 & x_2^4 + a_2 x_2^2 z_2^2 + z_2^4 \\
    y_3^2 & x_3^4 + a_3 x_3^2 z_3^2 + z_3^4 \\
    x_1 x_2 x_3 & c z_1 z_2 z_3
  \end{array}
  \right.
\end{equation}
with suitable parameters $a_1, a_2, a_3 \neq \pm 2$ and $c \neq 0$.
Here $x_j/z_j$ is the coordinate on~$E_j$ that changes sign under addition
of~$T_j$ (denoted~$x$ above).
We have to determine for which choices of the parameters $(a_1,a_2,a_3,c)$
we obtain isomorphic surfaces~$X/\Gamma$.

Let $A = \bigl(\BA^1 \setminus \{\pm 2\}\bigr)^3 \times \bigl(\BA^1 \setminus \{0\}\bigr)$
and consider $\calX \subset \P(1,2,1)^3 \times A$ as given by equations~\eqref{E:eqns},
where $x_j,y_j,z_j$ are the coordinates on the three factors~$\P(1,2,1)$
(for $j = 1,2,3$) and $a_1,a_2,a_3,c$ are the coordinates on~$A$.
Then the group of automorphisms of~$\P(1,2,1)^3 \times A$ that fix~$\calX$ and are compatible
with the fibration over~$A$ is generated by the following automorphisms,
which are specified by the image of the point
\[ \bigl((x_1:y_1:z_1), (x_2:y_2:z_2), (x_3:y_3:z_3), (a_1,a_2,a_3,c)\bigr) \,. \]
\[ \begin{array}{r@{\colon\quad}llll}
     \rho_1    & \bigl((ix_1:y_1:z_1), &(x_2:y_2:z_2),  &(x_3:y_3:z_3),  &(-a_1,a_2,a_3,ic)\bigr) \\
     \rho_2    & \bigl((x_1:y_1:z_1),  &(ix_2:y_2:z_2), &(x_3:y_3:z_3),  &(a_1,-a_2,a_3,ic)\bigr) \\
     \rho_3    & \bigl((x_1:y_1:z_1),  &(x_2:y_2:z_2),  &(ix_3:y_3:z_3), &(a_1,a_2,-a_3,ic)\bigr) \\
     \gamma'_1 & \bigl((x_1:-y_1:z_1), &(x_2:y_2:z_2),  &(x_3:y_3:z_3),  &(a_1,a_2,a_3,c)\bigr) \\
     \gamma'_2 & \bigl((x_1:y_1:z_1),  &(x_2:-y_2:z_2), &(x_3:y_3:z_3),  &(a_1,a_2,a_3,c)\bigr) \\
     \gamma'_3 & \bigl((x_1:y_1:z_1),  &(x_2:y_2:z_2),  &(x_3:-y_3:z_3), &(a_1,a_2,a_3,c)\bigr) \\
     \tau      & \bigl((z_1:y_1:x_1),  &(z_2:y_2:x_2),  &(z_3:y_3:x_3),  &(a_1,a_2,a_3,1/c)\bigr) \\
     \sigma    & \bigl((x_2:y_2:z_2),  &(x_3:y_3:z_3),  &(x_1:y_1:z_1),  &(a_2,a_3,a_1,c)\bigr) \\
     \tau'     & \bigl((x_2:y_2:z_2),  &(x_1:y_1:z_1),  &(x_3:y_3:z_3),  &(a_2,a_1,a_3,c)\bigr)
   \end{array}
\]
Note that
$\Gamma = \langle \rho_1^2 \rho_2^2 \gamma'_2, \rho_2^2 \rho_3^2 \gamma'_3,
                  \rho_3^2 \rho_1^2 \gamma'_1 \rangle$
fixes $\calX$ fiber-wise. Since we are interested in the quotients of the fibers by this
action, we restrict to the subgroup consisting of elements normalizing~$\Gamma$.
This means that we exclude~$\tau'$.

So let $\tilde{\Gamma} = \langle \rho_1, \rho_2, \rho_3, \gamma'_1, \gamma'_2, \gamma'_3, \tau, \sigma \rangle$.
The first six elements generate an abelian normal subgroup isomorphic to
$(\Z/4\Z)^3 \times (\Z/2\Z)^3$ on which $\tau$ acts by negation and $\sigma$ acts
by cyclically permuting the factors in both $(\Z/4\Z)^3$ and~$(\Z/2\Z)^3$.
In particular, $\tilde{\Gamma}$ has order $4^3 \cdot 2^3 \cdot 2 \cdot 3 = 2^{10} \cdot 3$.
We note that the kernel of the natural homomorphism $\tilde{\Gamma} \to \Aut(A)$
is $\tilde{\Gamma}_0 \colonequals \langle \Gamma, \gamma'_1, \gamma'_2 \rangle$
(we also have $\gamma'_3 = \gamma \gamma'_1 \gamma'_2 \in \tilde{\Gamma}_0$).
So on the quotient~$S$ of every fiber~$X$ by~$\Gamma$, we have an action
of~$\tilde{\Gamma}_0/\Gamma$, and the quotient of~$S$ by this action is the
del Pezzo surface~$Y$ in Burniat's construction.

The image~$\tilde{\Gamma}_A$ of~$\tilde{\Gamma}$ in~$\Aut(A)$ is generated by
\begin{align*}
  \bar{\rho}_1 \colon (a_1, a_2, a_3, c) &\longmapsto (-a_1, a_2, a_3, ic) \\
  \bar{\rho}_2 \colon (a_1, a_2, a_3, c) &\longmapsto (a_1, -a_2, a_3, ic) \\
  \bar{\rho}_3 \colon (a_1, a_2, a_3, c) &\longmapsto (a_1, a_2, -a_3, ic) \\
  \bar{\tau}   \colon (a_1, a_2, a_3, c) &\longmapsto (a_1, a_2, a_3, 1/c) \\
  \bar{\sigma} \colon (a_1, a_2, a_3, c) &\longmapsto (a_2, a_3, a_1, c)
\end{align*}
Note that $\zeta \colonequals \bar{\rho}_1^2 = \bar{\rho}_2^2 = \bar{\rho}_3^2$
just changes the sign of~$c$. The order of~$\tilde{\Gamma}_A$ is~$96$.
There are twenty conjugacy classes. The table below gives a representative element,
its order and the size of each class.
\[ \begin{array}{|r|cccccccccc|} \hline
     \text{rep.\large\strut} & 1 & \zeta & \bar{\rho}_1 & \bar{\rho}_1 \bar{\rho}_2 &
       \bar{\rho}_1 \bar{\rho}_2^{-1} & \bar{\rho}_1 \bar{\rho}_2 \bar{\rho_3} &
       \bar{\sigma} & \zeta \bar{\sigma} & \bar{\rho}_1 \bar{\sigma} & \bar{\sigma}^2 \\
     \text{order} & 1 & 2 & 4 & 2 & 2 & 4 & 3 & 6 & 12 & 3 \\
     \text{size}  & 1 & 1 & 6 & 3 & 3 & 2 & 4 & 4 &  8 & 4 \\\hline \hline
     \text{rep.\large\strut} & \zeta \bar{\sigma}^2 & \bar{\rho}_1 \bar{\sigma}^2 &
       \bar{\tau} & \bar{\rho}_1 \bar{\tau} & \bar{\rho}_1 \bar{\rho}_2 \bar{\tau} &
       \bar{\rho}_1 \bar{\rho}_2 \bar{\rho}_3 \bar{\tau} & \bar{\sigma} \bar{\tau} &
       \bar{\rho}_1 \bar{\sigma} \bar{\tau} & \bar{\sigma}^2 \bar{\tau} &
       \bar{\rho}_1 \bar{\sigma}^2 \bar{\tau} \\
     \text{order} & 6 & 12 & 2 & 2 & 2 & 2 & 6 & 6 & 6 & 6 \\
     \text{size}  & 4 &  8 & 2 & 6 & 6 & 2 & 8 & 8 & 8 & 8 \\\hline
   \end{array}
\]

Given a primary Burniat surface~$S$, the $(\Z/2\Z)^3$-cover $\pi \colon X \to S$
with $X \subset E_1 \times E_2 \times E_3$ is uniquely determined up to isomorphism
as the unramified covering corresponding to the unique normal abelian subgroup
of~$\pi_1(S)$ such that the quotient is~$(\Z/2\Z)^3$.
(See~\cite{burniat1}*{Section~3} for an explicit description of~$\pi_1(S)$,
which implies this.)
It follows that the (coarse) moduli space~$M$ of primary Burniat surfaces
is $A'/\tilde{\Gamma}_A$, where $A'$ is the open subset of~$A$ over which
the fibers of~$\calX$ do not pass through a fixed point of~$\gamma$.
We determine the invariants for this group action. We have the normal subgroup
$\langle \zeta, \bar{\tau} \rangle \cong (\Z/2\Z)^2$ that fixes $(a_1, a_2, a_3)$
and has orbit $\{c, -c, 1/c, -1/c\}$ on~$c$. The invariants for this action
are $a_1, a_2, a_3$ and~$v' \colonequals c^2 + 1/c^2$.
The action of~$\bar{\rho}_j$ on these invariants is given by changing the
signs of~$a_j$ and~$v'$, so the invariants of the action by the group
generated by the~$\bar{\rho}_j$ and~$\bar{\tau}$ are $a_1^2$, $a_2^2$, $a_3^2$,
$v \colonequals {v'}^2$ and $w \colonequals a_1 a_2 a_3 v'$. Finally,
$\bar{\sigma}$ fixes $v$ and~$w$ and permutes $a_1^2, a_2^2, a_3^2$ cyclically.
So the invariants of the action of~$\tilde{\Gamma}_A$ are the
elementary symmetric polynomials in the~$a_j^2$, namely
\begin{gather*}
  u_1 = a_1^2 + a_2^2 + a_3^2, \quad
  u_2 = a_1^2 a_2^2 + a_2^2 a_3^2 + a_3^2 a_1^2, \quad
  u_3 = a_1^2 a_2^2 a_3^2, \\
  \text{together with} \quad d = (a_1^2 - a_2^2)(a_2^2 - a_3^2)(a_3^2 - a_1^2)
  \quad \text{and $v$, $w$.}
\end{gather*}

This exhibits the moduli space~$M$ as an open subset of a bidouble cover of~$\BA^4$
(with coordinates $u_1$, $u_2$, $u_3$, $v$) as follows.

\begin{Theorem}\label{modulispace}
  The moduli space~$M$ of primary Burniat surfaces is the subset
  of~$\BA^6$ defined by the following equations and inequalities.
  \begin{align*}
    -4 u_1^3 u_3 + u_1^2 u_2^2 + 18 u_1 u_2 u_3 - 4 u_2^3 - 27 u_3^2 &= d^2 \\
    u_3 v &= w^2 \\
    (v - u_1)^2 + u_2 (v - 4) + u_3 + (u_1 + v - 8) w &\neq 0 \\
    64 - 16 u_1 + 4 u_2 - u_3 &\neq 0
  \end{align*}
\end{Theorem}

\begin{proof}
  This follows from the discussion above. Note that the expression on the
  left hand side of the first equation is the discriminant
  of $X^3 - u_1 X^2 + u_2 X - u_3$.
  The first inequality is equivalent to the non-vanishing of the discriminant of~$X$
  as defined in Section~\ref{S:explicit} below and the second inequality encodes the
  non-singularity of the curves $y_j^2 = x_j^4 +a_j x_j^2 z_j^2 + z_j^4$
  (which is equivalent to $a_j^2 \neq 4$).
\end{proof}

\begin{Corollary}
  $M$ is rational.
\end{Corollary}

\begin{proof}
  First note that $M$ is birational to the product of~$\BA^1$ with the double
  cover of~$\BA^3$ given by the first equation (we can eliminate $v$ using the
  second equation, and then $w$ is a free variable).
  The first equation is equivalent to
  \[ 4 (3 u_2 - u_1^2)^3 = (27 u_3 + 2 u_1^3 - 9 u_1 u_2)^2 + 27 d^2 \,. \]
  By a coordinate change, we reduce to the product of~$\BA^2$ with the rational
  surface given by $x^3 + y^2 + z^2 = 0$.
\end{proof}


\section{Automorphism groups of primary Burniat surfaces} \label{S:auto}

We now consider the group of automorphisms of a primary Burniat surface~$S$.
Since the covering $\pi \colon X \to S$ is unique, every automorphism of~$S$
lifts to an element of~$\tilde{\Gamma}$ fixing~$X$. We have already seen
that $\Aut(S)$ always contains
$\bar{\Gamma}_0 \colonequals \tilde{\Gamma}_0/\Gamma \cong (\Z/2\Z)^2$;
since $\tilde{\Gamma}_0$ is normal in~$\tilde{\Gamma}$, this is a normal
subgroup of~$\Aut(S)$.

Additional automorphisms of~$X$ correspond to elements of~$\tilde{\Gamma}$ that fix the
parameters~$(a_1, a_2, a_3, c)$. There are the following possibilities
for such elements modulo~$\tilde{\Gamma}_0$. We give their images in~$\tilde{\Gamma}_A$,
up to conjugation. Note that $\zeta$, $\bar{\rho}_1$, $\bar{\rho}_1 \bar{\rho}_2 \bar{\rho}_3$,
$\bar{\rho}_1 \bar{\rho}_2$, $\zeta \bar{\sigma}$, $\bar{\rho}_1 \bar{\sigma}$,
$\zeta \bar{\sigma}^2$, $\bar{\rho}_1 \bar{\sigma}^2$ map $c$ to $\pm i c$ or~$-c$
and cannot fix any $c \neq 0$. Also, $\bar{\rho}_1 \bar{\rho}_2 \bar{\rho}_3 \bar{\tau}$
can only fix~$(a_1,a_2,a_3,c)$ when $a_1 = a_2 = a_3 = 0$ and $c^4 + 1 = 0$,
which is excluded (it corresponds to a singular~$X$).
\begin{enumerate}[(i)]
  \item $\bar{\tau}$, corresponding to $c = 1/c$; this is $v = 4$ on~$M$. \\
        Then $\Aut(S)$ acquires an extra involution~$\tau'$ that commutes
        with~$\bar{\Gamma}_0$. We write $M_1$ for the corresponding subvariety of~$M$.
  \item $\bar{\rho}_1 \bar{\rho}_2^{-1}$, corresponding to $a_1 = a_2 = 0$;
        this is $u_2 = u_3  = 0$ on~$M$, which implies $d = w = 0$. \\
        Then $\Aut(S)$ acquires an element~$\rho'$ of order~$4$ that commutes
        with~$\bar{\Gamma}_0$ and whose square is in~$\bar{\Gamma}_0$.
        We write $M_2$ for the corresponding subvariety of~$M$.
  \item $\bar{\rho}_1 \bar{\tau}$, corresponding to $a_1 = 0$ and $c = i/c$;
        this is $u_3 = v = 0$ on~$M$, which implies $w = 0$. \\
        Then $\Aut(S)$ acquires an extra involution~$\tau''$ that commutes
        with~$\bar{\Gamma}_0$. We write $M_3$ for the corresponding subvariety of~$M$.
  \item $\bar{\sigma}$, corresponding to $a_1 = a_2 = a_3$;
        this is $d = 0$, $u_1^2 = 3 u_2$, $u_1 u_2 = 9 u_3$ on~$M$. \\
        Then $\Aut(S)$ acquires an extra element~$\sigma'$ of order~$3$ that acts
        non-trivially on~$\bar{\Gamma}_0$. The subgroup of~$\Aut(S)$ generated
        by~$\bar{\Gamma}_0$ and the additional automorphism is isomorphic to the
        alternating group~$A_4$. We write $M_4$ for the corresponding subvariety of~$M$.
\end{enumerate}
By considering the action on~$(a_1, a_2, a_3, c)$, it can be checked that elements
of other conjugacy classes can only be present when we are in one of the four
cases (i)--(iv) above. More precisely, of the $14$~subgroups up to conjugacy
that do not contain one of the `forbidden' elements given above, four cannot
occur, since the action implies that there is actually a larger group present.
The ten remaining groups correspond to those given below.

Note that $M_1 \cap M_3 = M_3 \cap M_4 = \emptyset$. We obtain the following
stratification of~$M$. The leftmost column gives the dimension.
\[ \xymatrix{ 4 & & & & M \ar@{-}[lldd] \ar@{-}[ldd] \ar@{-}[d] \ar@{-}[ddr] \\
              3 & & & & M_1 \ar@{-}[ldd] \ar@{-}[rdd] \\
              2 & & M_3 \ar@{-}[d]
                  & M_2 \ar@{-}[ld] \ar@{-}[d] \ar@{-}[rd]
                  &
                  & M_4 \ar@{-}[ld] \ar@{-}[d] \\
              1 & & M_2 \cap M_3
                  & M_1 \cap M_2 \ar@{-}[rd]
                  & M_2 \cap M_4 \ar@{-}[d]
                  & M_1 \cap M_4 \ar@{-}[ld] \\
              0 & & & & M_1 \cap M_2 \cap M_4
            }
\]
Note that $M_1 \cap M_2 \cap M_4$ is the single point
$(u_1,u_2,u_3,d,v,w) = (0,0,0,0,4,0)$, which is represented by
$(a_1,a_2,a_3,c) = (0,0,0,1)$.

The following table specifies $\Aut(S)$ and $\Aut(S)/\bar{\Gamma}'$
for the various strata. We write $C_n$ for a cyclic group
of order~$n$ and $D_4$ for the dihedral group of order~$8$.
For a group~$G$, $(G \wr C_3)/G$ denotes the wreath product $G^3 \rtimes C_3$
divided by the diagonal image of~$G$ in the normal subgroup~$G^3$.
Note that $A_4 \cong (C_2 \wr C_3)/C_2$.
The semidirect product $(C_4 \wr C_3)/C_4 \rtimes C_2$ in the last line
is such that the nontrivial element of~$C_2$ (which is~$\tau'$)
acts on $(C_4 \wr C_3)/C_4$
by simultaneously inverting the elements in the $C_4$~components.
The table can be checked by observing that ${\tau'}^2 = {\tau''}^2 = 1$,
${\rho'}^2 \in \bar{\Gamma}_0$, ${\rho'}^4 = 1$, $\tau' \rho' = {\rho'}^{-1} \tau'$,
$\tau'' \rho' = {\rho'}^{-1} \tau''$, $\tau' \sigma' = \sigma' \tau'$,
$(\rho' \sigma')^3 = 1$, and $\tau'$, $\rho'$, $\tau''$ commute with~$\bar{\Gamma}_0$,
whereas $\sigma'$ acts non-trivially on it.
\[ \renewcommand{\arraystretch}{1.1}
   \begin{array}{|c|c|c|} \hline
     \text{stratum\large\strut} & \Aut(S)      & \Aut(S)/\bar{\Gamma}' \\\hline
     M                     & C_2^2             & C_1 \text{\large\strut} \\
     M_1                   & C_2^3             & C_2 \\
     M_2                   & C_2 \times C_4    & C_2 \\
     M_3                   & C_2^3             & C_2 \\
     M_4                   & A_4               & C_3 \\
     M_1 \cap M_2          & C_2 \times D_4    & C_2^2 \\
     M_1 \cap M_4          & C_2 \times A_4    & C_6 \\
     M_2 \cap M_3          & C_2 \times D_4    & C_2^2 \\
     M_2 \cap M_4          & (C_4 \wr C_3)/C_4 & A_4 \\
     M_1 \cap M_2 \cap M_4 & (C_4 \wr C_3)/C_4 \rtimes C_2 & A_4 \times C_2 \\\hline
   \end{array}
\]



\section{Low genus curves on primary Burniat surfaces} \label{S:low}

We now study the set of curves of geometric genus $0$ or~$1$ on a primary
Burniat surface~$S$. This is more easily done using Inoue's construction
of~$S$ as an \'etale quotient of a hypersurface~$X$ in a product $E_1 \times E_2 \times E_3$
of three elliptic curves. Since $X$ does not contain rational curves, the same
is true for~$S$. So it remains to find the curves of geometric genus~$1$ on~$S$.
Such curves are images of smooth curves of genus~$1$ contained in~$X$
under the quotient map $\pi \colon X \to S$, so we need to classify these.

Let $E \subset X$ be a curve of genus~$1$. The three projections of
$E_1 \times E_2 \times E_3$ give us three morphisms $\psi_j \colon E \to X \to E_j$,
which must be translations followed by isogenies or constant.
There are three cases. Note that at least one of these morphisms must be
non-constant.

\subsection{Two of the $\psi_j$ are constant} \label{SS:2const}

Say $\psi_2$ and~$\psi_3$ are constant, with values $(\xi_2 : \eta_2 : \zeta_2)$
and $(\xi_3 : \eta_3 : \zeta_3)$. If $\xi_2 \xi_3 \neq 0$, this would force
$x_1 = c (\zeta_2 \zeta_3/\xi_2 \xi_3) z_1$. and $\psi_1$ would have to be
constant as well. So $\xi_2 = 0$ (say). Since $c \neq 0$, we must then have
$\zeta_3 = 0$. This gives two choices for the image of~$\psi_2$ and two
choices for the image of~$\psi_3$; by interchanging the roles of $\xi$ and~$\zeta$,
we obtain four further choices. This gives rise to eight copies of~$E_1$ in~$X$
that are pairwise disjoint. Under the action of~$\Gamma$, they form two orbits
of size four, so on~$S$ we obtain two copies of $E_1/\langle T_1 \rangle$.
In the same way, we have two copies each of $E_2/\langle T_2 \rangle$
and of $E_3/\langle T_3 \rangle$. These six smooth curves of genus~$1$
are arranged in the form of a hexagon, with the two images of~$E_j$ corresponding
to opposite sides, see Figure~\ref{Fighex}.
In Burniat's original construction, these curves are
obtained as preimages of the three lines joining the three points that are blown up
and the three exceptional divisors obtained from blowing up the points.

\begin{Definition}
  We call these six curves the \emph{curves at infinity} on~$S$.
\end{Definition}

\begin{figure}[htb]
  \hrulefill
  \begin{center}
    \Gr{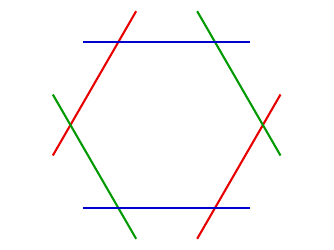}{0.6\textwidth}
  \end{center}
  \caption{The six curves at infinity.}
  \label{Fighex}
  \hrulefill
\end{figure}

\subsection{Exactly one of the $\psi_j$ is constant} \label{SS:1const}

Say $\psi_1$ is constant. The fibers of any of the projections $X \to E_j$
are curves of genus~$5$, which are generically smooth. Taking $j = 1$,
we get explicit equations for the fiber in the form
\begin{align*}
  y_2^2 &= x_2^4 + a_2 x_2^2 z_2^2 + z_2^4 \\
  y_3^2 &= x_3^4 + a_3 x_3^2 z_3^2 + z_3^4 \\
  \xi_1 x_2 x_3 &= c \zeta_1 z_2 z_3
\end{align*}
where $(\xi_1 : \eta_1 : \zeta_1) \in E_1$ is the point we take the fiber over.
If $\xi_1 = 0$ or $\zeta_1 = 0$, then the fiber splits into two of the curves
described in Section~\ref{SS:2const} (taken twice). Otherwise, the third equation
is equivalent to $(x_3 : z_3) = (c \zeta_1 z_2 : \xi_1 x_2)$. Using this
in the second equation, we can reduce to the pair
\begin{align*}
  y_2^2 &= x_2^4 + a_2 x_2^2 z_2^2 + z_2^4 \\
  y_3^2 &= \xi_1^4 x_2^4 + a_3 c^2 \xi_1^2 \zeta_1^2 x_2^2 z_2^2 + c^4 \zeta_1^4 z_2^4
\end{align*}
of equations describing the fiber as a bidouble cover of~$\P^1$.

The fiber degenerates if and only if the two quartic forms in $(x_2, z_2)$
have common roots. This occurs exactly when their resultant
\[ R = \left(c^4 \xi_1^4 \zeta_1^4 (a_2^2 + a_3^2)
              - c^2 \xi_1^2 \zeta_1^2 (\xi^4 + c^4 \zeta_1^4) a_2 a_3
              + (\xi_1^4 - c^4 \zeta_1^4)^2\right)^2
\]
vanishes. Writing $\xi = \xi_1/(c \zeta_1)$, this is equivalent to
\[ \xi^8 - a_2 a_3 \xi^6 + (a_2^2 + a_3^2 - 2) \xi^4 - a_2 a_3 \xi^2 + 1 = 0 \,.\]
So generically (in terms of $a_2$ and~$a_3$) this occurs for eight values of~$\xi$,
and then the two quartic forms have one pair of common roots, leading to
a fiber that is a curve of geometric genus~$3$ with two nodes. But it is also
possible that both quartics are proportional. This happens if and only if
$a_2 = \pm a_3$ (and $\xi = \pm 1$ if $a_2 = a_3$, $\xi = \pm i$ if $a_2 = -a_3$).
In that case, the fiber splits into two curves isomorphic to $E_2$ and~$E_3$
that meet transversally in the four points where $y_2 = y_3 = 0$.
The two values of~$\xi$ give rise to four fibers (note that the point
$(\xi_1 : \eta_1 : \zeta_1)$ on~$E_1$ cannot have $\eta_1 = 0$, since then
$X$ would contain a fixed point of~$\gamma$), together containing eight copies
of $E_2$ (or~$E_3$). The action of~$\Gamma$ permutes these eight curves transitively.
The stabilizer of a fiber swaps its two components and interchanges the four
intersection points in pairs, so that the image of either component on~$S$
is a curve of geometric genus~$1$ with two nodes. Note that the existence
of these curves does not depend on the value of~$c$ (as long as $X$ is smooth).

If both $a_2 = a_3$ and $a_2 = -a_3$, so that $a_2 = a_3 = 0$ (which means that
the curves have complex multiplication by~$\Z[i]$), then we obtain two orbits
of this kind, leading to two curves of geometric genus~$1$ on~$S$.
So in the extreme case $a_1 = a_2 = a_3 = 0$, we obtain the maximal number
of six such curves on~$S$. This situation is shown in Figure~\ref{Fig12}
in terms of the images of the curves in~$\P^2$ with respect to Burniat's construction.
In general, the table below gives the number
of these curves (the subscripts are understood up to cyclic permutation).
\[ \begin{array}{|c|c|} \hline
     \text{condition} & \text{number of curves} \\\hline
     a_1^2 \neq a_2^2 \neq a_3^2 \neq a_1^2 & 0 \text{\large\strut}\\
     0 \neq a_1^2 = a_2^2 \neq a_3^2        & 1 \\
     0 = a_1 = a_2 \neq a_3                 & 2 \\
     0 \neq a_1^2 = a_2^2 = a_3^2           & 3 \\
     0 = a_1 = a_2 = a_3                    & 6 \\\hline
   \end{array}
\]

Note that the fibrations of~$X$ given by the projections to the~$E_j$ are
stable under~$\Gamma$. They induce fibrations $S \to \P^1$ with fibers of
genus~$3$. In Burniat's construction, these fibrations are obtained by
pulling back the pencils of lines through one of the three points on~$\P^2$
that are blown up. Using this, it is easy to see that the fibers in~$X$ over the
ramification points of $E_j \to \P^1$ map to smooth curves of genus~$2$
on~$S$ (taken twice).

\begin{figure}[htb]
  \hrulefill
  \begin{center}
    \Gr{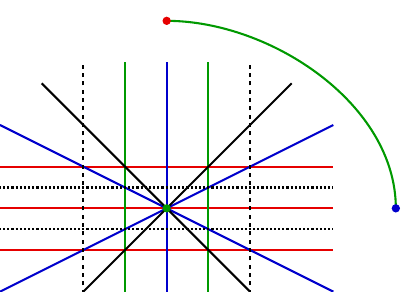}{0.6\textwidth}
  \end{center}
  \caption{Six curves of type~I. They are the preimages in~$S$
           of the six black lines. Here two of the blown-up points are at infinity.}
  \label{Fig12}
  \hrulefill
\end{figure}

\begin{Definition}
  We call the curves of geometric genus~$1$ on~$S$ arising in this way
  \emph{curves of type~I}.
\end{Definition}

See the left part of Figure~\ref{Figtypes} for how a curve of type~I
sits inside~$S$.

We write $N_1$ for the subset of~$M$ given by $d = 0$ (so that two of the~$a_j$
agree up to sign, compare Theorem~\ref{modulispace}), which corresponds
to surfaces admitting curves of type~I.

\subsection{No $\psi_j$ is constant} \label{SS:0const}

With a slight abuse of notation we write $x_j$ for the rational function~$x_j/z_j$
on~$E_j$ (and~$X$). If $E \subset X$ is an elliptic curve and all~$\psi_j$
are non-constant, then each~$\psi_j$ has the form $\psi_j(P) = \varphi_j(P - Q_j)$,
where $Q_j \in E$ is some point and $\varphi_j \colon E \to E_j$ is
an isogeny. The condition that $E$ is contained in~$X$ then means that
\[ (x_1 \circ \psi_1) \cdot (x_2 \circ \psi_2) \cdot (x_3 \circ \psi_3) = c \,. \]
In particular, we must have
\begin{align*}
  \tau_{Q_1} \varphi_1^*\bigl(\div(x_1)\bigr)
    &+ \tau_{Q_2} \varphi_2^*\bigl(\div(x_2)\bigr)
     + \tau_{Q_3} \varphi_3^*\bigl(\div(x_3)\bigr) \\
  &= \div(x_1 \circ \psi_1) + \div(x_2 \circ \psi_2) + \div(x_3 \circ \psi_3)
   = 0 \,,
\end{align*}
where $\tau_Q$ denotes translation by~$Q$.

It can be checked that the four points above $0$ and~$\infty$ on~$E_j$
form a principal homogeneous space for the $2$-torsion subgroup~$E_j[2]$.
So, taking the coefficients mod~$2$, $\div(x_j) \equiv [2]^*(P_j)$ for some
point $P_j \in E_j$, where $[2]$ denotes the multiplication-by-$2$ map.
The relation above then implies that
\[ \tau_{Q'_1} [2]^* K_1 + \tau_{Q'_2} [2]^* K_2 + \tau_{Q'_3} [2]^* K_3 \equiv 0 \bmod 2 \]
with suitable points $Q'_j \in E$, where $K_j$ denotes the formal sum of the points
in the kernel of~$\varphi_j$. This in turn is equivalent to
\[ K_1 + \tau_{2Q'_2-2Q'_1} K_2 + \tau_{2Q'_3-2Q'_1} K_3 \equiv 0 \bmod 2 \,. \]
In down-to-earth terms, this means that there are cosets of $\ker(\varphi_1)$
(which we can take to be $\ker(\varphi_1)$ itself), $\ker(\varphi_2)$
and~$\ker(\varphi_3)$ such that every point of~$E$
is contained in either none or exactly two of these cosets.
Since $0 \in \ker(\varphi_1)$, exactly one of the two other cosets must contain~$0$;
we can assume that this is~$\ker(\varphi_2)$. Then the condition is that the symmetric
difference of $\ker(\varphi_1)$ and~$\ker(\varphi_2)$ must be a coset
of~$\ker(\varphi_3)$.

Since $E \subset X \subset E_1 \times E_2 \times E_3$, the intersection
of all three kernels is trivial.
Now assume that $G = \ker(\varphi_1) \cap \ker(\varphi_2) \neq \{0\}$.
Then the symmetric difference of $\ker(\varphi_1)$ and~$\ker(\varphi_2)$
is a union of cosets of~$G$, which implies that $G \subset \ker(\varphi_3)$,
a contradiction. So $\ker(\varphi_1) \cap \ker(\varphi_2) = \{0\}$
(and the same holds for the other two intersections)
and hence the symmetric difference is
$\ker(\varphi_1) \cup \ker(\varphi_2) \setminus \{0\}$.
Now if $\#\ker(\varphi_1) > 2$, then this symmetric difference contains
two elements of~$\ker(\varphi_1)$ whose difference then must be
in~$\ker(\varphi_1) \cap \ker(\varphi_3)$, which is impossible.
So $\#\ker(\varphi_1) \le 2$ and similarly $\#\ker(\varphi_2) \le 2$.
On the other hand, the kernels cannot both be trivial (then the symmetric
difference would be empty and could not be a coset). So, up to cyclic permutation
of the subscripts, this leaves two possible scenarios:
\begin{enumerate}[(1)]
  \item $\#\ker(\varphi_1) = 2$ and $\ker(\varphi_2)$ and~$\ker(\varphi_3)$ are trivial;
  \item $\ker(\varphi_1)$, $\ker(\varphi_2)$ and $\ker(\varphi_3)$ are the three
        subgroups of order~$2$ of~$E$.
\end{enumerate}
We now have to check whether and how these possibilities can indeed be realized.

\begin{figure}[htb]
  \hrulefill
  \begin{center}
    \Gr{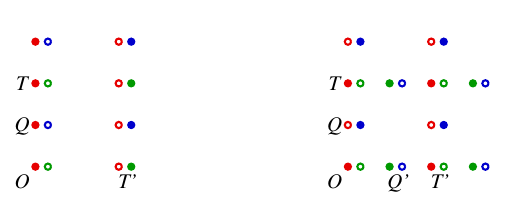}{0.8\textwidth}
  \end{center}
  \caption{The possible arrangements of the divisors of $x_1 \circ \psi_1$ (red),
           $x_2 \circ \psi_1$ (green) and $x_3 \circ \psi_3$ (blue).
           On the left is Scenario~(1), on the right Scenario~(2).
           $\bullet$ denotes a zero, $\circ$ a pole of $x_j \circ \psi_j$.}
  \label{Figker}
  \hrulefill
\end{figure}

In Scenario~(1), $E$, $E_2$ and~$E_3$ are isomorphic.
We identify $E_2$ and~$E_3$ with~$E$ via the isomorphisms $\psi_2$ and~$\psi_3$.
The points above $0$ and~$\infty$ on $E_2$ and~$E_3$ then have to be distinct,
and the two corresponding cosets of~$E[2]$ must differ by a point~$Q$ of order~$4$
such that $\ker(\varphi_1) = \{0, 2Q\}$. In $\div(x_1 \circ \psi_1)$,
points differing by~$2Q$ then have the same sign, which implies that
$T_2 = T_3 = 2Q \mathrel{=:} T$ (recall that $T_j$ is the point of order~$2$ on~$E_j$
such that adding~$T_j$ changes the sign of~$x_j$). Let $T' \in E$ be a point
of order~$2$ with $T' \neq T$. Then we can realize this scenario by taking
\[ \psi_1 \colon E \to E_1 \colonequals E/\langle T \rangle, \quad
   \psi_2 = \id_E \colon E \to E_2 \colonequals E, \quad
   \psi_3 = \id_E \colon E \to E_3 \colonequals E
\]
and choosing the~$x_j$ such that
\begin{align*}
   \div(x_1 \circ \psi_1) &= (O) + (Q) + (T) + (-Q) - (T') - (T'+Q) - (T'+T) - (T'-Q)\,, \\
   \div(x_2 \circ \psi_2) &= (T') + (T+T') - (O) - (T)\,, \\
   \div(x_3 \circ \psi_3) &= (T'+Q) + (T'-Q) - (Q) - (-Q)\,.
\end{align*}
See the left part of Figure~\ref{Figker}.

More concretely, let
\[ E \colon y^2 = x(x-1)(x-\lambda^2) \,, \]
which has a point $Q = (\xi,\eta) = (\lambda, i \lambda(1 - \lambda ))$ of
order~$4$ such that $T = 2Q = (0,0)$. Then $E_1 = E/\langle T \rangle$ is
\[ E_1 \colon y^2 = (x - 2\lambda)(x + 2\lambda)(x - (1+\lambda^2)) \,; \]
with
\[ x_1 = \frac{(1+\lambda)y}{(x + 2\lambda)(x - (1+\lambda^2))} \quad\text{and}\quad
   y_1 = \frac{x^2 - 4\lambda x + 4\lambda(1-\lambda+\lambda^2)}{(x+2\lambda)(x-(1+\lambda^2))}
\]
this gives the equation
\[ y_1^2 = x_1^4 + 2\frac{1- 6\lambda + \lambda^2}{(1+\lambda)^2} x_1^2 + 1 \]
for~$E_1$. On $E_2 = E = E_3$, we take
\[ x_2 = \frac{1}{\sqrt{1-\lambda^2}} \, \frac{y}{x} \qquad\text{and}\qquad
   x_3 = \sqrt{-\frac{1-\lambda}{1+\lambda}} \, \frac{x+\lambda}{x-\lambda} \,.
\]
This gives the following equations for $E_2$ and~$E_3$:
\[ y_2^2 = x_2^4 + 2\frac{1+\lambda^2}{1-\lambda^2} x_2^2 + 1, \qquad
   y_3^2 = x_3^4 + 2\frac{1+\lambda^2}{1-\lambda^2} x_3^2 + 1
\]
with $y_2 = 1/(1-\lambda^2) \cdot (x - \lambda^2/x)$
and $y_3 = 4i\lambda/(1+\lambda) \cdot y/(x-\lambda)^2$.
Since the isogeny $E \to E_1$ is given by
$(x,y) \mapsto (x + \lambda^2/x, (1 - \lambda^2/x^2) y)$,
we get indeed that
\begin{align*}
  (x_1 \circ \psi_1) x_2 x_3
    &= \frac{(1+\lambda)(x-\lambda) y}{(x+\lambda)(x-1)(x-\lambda^2)}
       \cdot \frac{1}{\sqrt{1-\lambda^2}} \, \frac{y}{x}
       \cdot \sqrt{-\frac{1-\lambda}{1+\lambda}} \, \frac{x+\lambda}{x-\lambda} \\
    &= \pm i \frac{y^2}{x(x-1)(x-\lambda^2)} = \pm i
\end{align*}
is constant. (The sign depends on the choice of the two square roots.)

We can read off the intersection multiplicities with the elliptic curves
`at infinity' of~$X$ from the divisors of the functions $x_j \circ \psi_j$:
a point that occurs positively in~$\div(x_j \circ \psi_j)$ and
negatively in~$\div(x_k \circ \psi_k)$ contributes one (transversal) intersection
of~$E$ with a curve in $X \cap \{x_j = z_k = 0\}$ and therefore also contributes~1
to the intersection multiplicity of $\pi(E)$ with the image of these
curves on~$S$. In our case, we obtain intersection multiplicity~$2$ for
\[ x_1 = z_2 = 0, \quad x_1 = z_3 = 0, \quad z_1 = x_2 = 0, \quad z_1 = x_3 = 0, \]
and no intersection for $x_2 = z_3 = 0$ and for $z_2 = x_3 = 0$.

\begin{figure}[htb]
  \hrulefill
  \begin{center}
    \strut\hfill\Gr{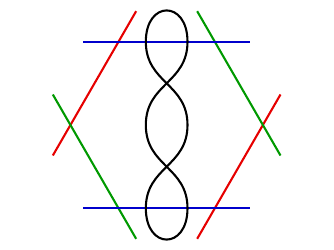}{0.45\textwidth} \hfill \Gr{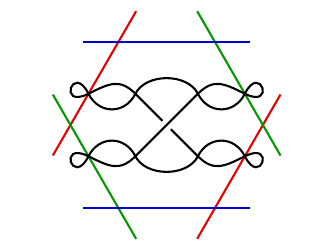}{0.45\textwidth}\hfill\strut
  \end{center}
  \caption{A curve of type~I (left) and a curve of type~II (right) sitting inside~$S$.}
  \label{Figtypes}
  \hrulefill
\end{figure}

The $\Gamma$-orbit of such a curve has size~$8$. In terms of points~$P$ on~$E$,
the action of~$\Gamma$ on~$E$ is as follows. We write $\psi = \psi_1$.
Note that the automorphisms that negate the $y$-coordinate are given,
respectively, by $P \mapsto \psi(Q) - P$, $P \mapsto T - P$, $P \mapsto -P$
on $E_1$, $E_2$, $E_3$.
\[ \renewcommand{\arraystretch}{1.1}
   \begin{array}{|r|ccc|c|} \hline
                       &             E_1 & E_2 & E_3 & \text{intersection} \\\hline
                   \id &         \psi(P) &   P &   P & \text{same curve} \\
              \gamma_1 &         \psi(P) & T+P & T-P & P = O, T, T', T+T' \\
              \gamma_2 &        -\psi(P) &   P & T+P & \text{none} \\
              \gamma_3 & \psi(Q)+\psi(P) &  -P &   P & \text{none} \\
     \gamma_2 \gamma_3 & \psi(Q)-\psi(P) &  -P & T+P & \text{none} \\
     \gamma_1 \gamma_3 & \psi(Q)+\psi(P) & T-P & T-P & 2P = \pm Q \text{\ (8 points)} \\
     \gamma_1 \gamma_2 &        -\psi(P) & T+P &  -P & P = Q, -Q, T'+Q, T'-Q \\
                \gamma & \psi(Q)-\psi(P) & T-P &  -P & \text{none} \\\hline
   \end{array}
\]

In the column labeled `intersection', we list the points $P \in E$
that lie on the corresponding translate of~$E$.
The entries show that
each curve in the orbit intersects two other curves in the orbit in
four points each, which are on some of the curves at infinity, and intersects
another curve in eight points. These intersection points are swapped in pairs
by the element of~$\Gamma$
that interchanges the two intersecting curves. So the image of this orbit in~$S$
is a curve of geometric genus~$1$ with eight nodes (and hence arithmetic genus~$9$).
Its image on~$\P^2$ is a conic passing through exactly two of the blown-up points,
which at these points is tangent to one of the lines in the branch locus,
intersects the two lines joining the two points with the third at an intersection
point with another branch line and is tangent to two further branch lines;
see Figure~\ref{FigSc1}, Left. It can be checked that there can be at most one such
conic through a given pair of the points unless the two corresponding~$E_j$
have $a_j = 0$. In this case, there are either none or two such conics,
see Figure~\ref{FigSc1}, Right. (When $a_2 = a_3 = 0$, say, then
$(x_1, y_1, x_2, y_2, x_3, y_3) \mapsto (-x_1, y_1, ix_2, y_2, ix_3, y_3)$
is an additional automorphism of~$X$, which descends to an automorphism of~$S$
that swaps the two curves.)

If such a curve is to exist on~$S$, two of the $a_j$ must agree up to sign,
and the third must satisfy a relation with the first two, which is (for $a_2 = a_3$ as above)
\[ \bigl(a_1 - 2 (a_2^2 - 3)\bigr)^2 =  4 a_2^2 (a_2^2 - 4) \,. \]
In addition, we must have $c^2 = -1$ if $a_2 = a_3 \neq 0$, and $c^4 = 1$ otherwise.
Of course, we can replace $a_1$ by~$-a_1$; then $c^2 = 1$ if $a_2 = a_3 \neq 0$.
This translates into $(u_1, u_2, u_3, w)$ being on a curve (and $d = 0$, $v = 4$).

The locus in~$M$ of points whose associated surfaces contain curves of type~II
is a smooth rational curve of degree~$6$. One possible parameterization is
\begin{align*}
  u_1 &= -2 \frac{t^3 - 4 t^2 - 7 t - 8}{t^2} \\
  u_2 &= \frac{(t - 1)^2 (t^3 - 10 t^2 - 31 t - 32)}{t^3} \\
  u_3 &= 4 \frac{(t - 1)^4 (t + 2)^2}{t^4} \\
    d &= 0 \\
    v &= 4 \\
    w &= -4 \frac{(t - 1)^2 (t + 2)}{t^2}
\end{align*}
We write $N_2$ for this rational curve in~$M$.

There are four values of~$\lambda$ such that all~$a_j$ are equal up to sign;
this leads to $a_j = \pm\alpha$ with $\alpha = \frac{3}{8}(5 \pm \sqrt{-7})$ and
\[ (u_1,u_2,u_3,d,v,w) = (3\alpha^2, 3\alpha^4, \alpha^6, 0, 4, 2\alpha^3)\,. \]
In this case, $E$ has complex multiplication by the order of discriminant~$-7$,
and the isogeny $E \to E_1$ is multiplication by $(1 \pm \sqrt{-7})/2$.
So in this special situation, we can have three curves of this type.

\begin{figure}[htb]
  \hrulefill
  \begin{center}
    \Gr{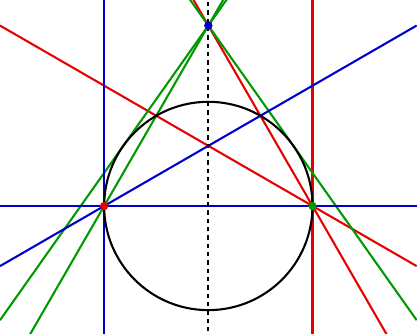}{0.4\textwidth} \hfill \Gr{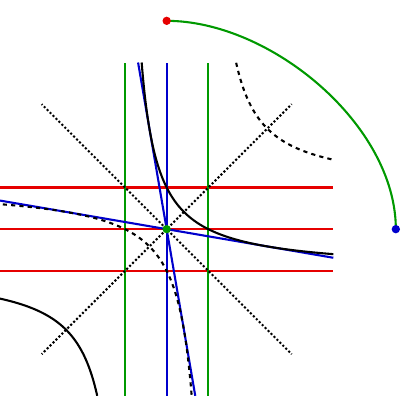}{0.4\textwidth}
  \end{center}
  \caption{\textbf{Left:} A curve of type~II. It is the preimage in~$S$
           of the black conic. The dashed line leads to a curve of type~I. \newline
           \textbf{Right:} When two of the curves have $a_j = 0$, then we can obtain two curves
           of type~II (given by the solid and the dashed black conics)
           and two curves of type~I (given by the two dotted lines).}
  \label{FigSc1}
  \hrulefill
\end{figure}

\begin{Definition}
  We call the curves of geometric genus~$1$ on~$S$ arising as in Scenario~(1)
  \emph{curves of type~II}.
\end{Definition}

See the right part of Figure~\ref{Figtypes} for how a curve of type~II
sits inside~$S$.

For scenario~(2), we consider $E$ with two points $Q$, $Q'$ of order~$4$
that span~$E[4]$. We set $T = 2Q$, $T' = 2Q'$
and write $D = (O) + (T) + (T') + (T+T')$ for the
divisor that is the sum of the $2$-torsion points on~$E$.
Then essentially the only way to realize this scenario is via
\begin{align*}
 \psi_1 &\colon E \To E_1 \colonequals E/\langle T \rangle
   &\text{with}\quad \div(x_1 \circ \psi_1) &= D - \tau_{Q}(D) \,, \\
 \psi_2 &\colon E \To E_2 \colonequals E/\langle T' \rangle
   &\text{with}\quad \div(x_2 \circ \psi_2) &= \tau_{Q'}(D) - D \,, \\
 \psi_3 &\colon E \To E_3 \colonequals E/\langle T+T' \rangle
   &\text{with}\quad \div(x_3 \circ \psi_3) &= \tau_{Q}(D) - \tau_{Q'}(D)\,.
\end{align*}
See the right part of Figure~\ref{Figker}.

We are now going to show that this scenario is impossible.
For this, note that the map negating the $y$-coordinate is given
on $E_1$, $E_2$ and~$E_3$ by
$P \mapsto \psi_1(T') - P$, $P \mapsto \psi_2(T) - P$ and $P \mapsto -P$,
respectively. Consider $P = Q + Q' \in E$. We have
\begin{align*}
  P &= \bigl(\psi_1(Q+Q'), \psi_2(Q+Q'), \psi_3(Q+Q')\bigr) \\
    &= \bigl(\psi_1(-Q+Q'), \psi_2(Q-Q'), \psi_3(-Q-Q')\bigr) \\
    &= \bigl(\psi_1(T')-\psi_1(P), \psi_2(T)-\psi_2(P), -\psi_3(P)\bigr) \\
    &= \gamma(P) \,,
\end{align*}
so that $\gamma$ would have $P$ as a fixed
point on~$X$, which contradicts our assumptions. So this scenario
is indeed impossible.

\begin{table}[htb]
  \hrulefill

  \begin{minipage}{0.52\textwidth}
  \[ \renewcommand{\arraystretch}{1.1}
     \begin{array}{|c|c|c|c|c|c|} \hline
       \text{\large\strut}
        \in & \notin        & \infty & \,\text{I}\, & \text{II} & \text{dim.} \\\hline
       \text{\large\strut}
        M   & N_1           &      6 &            0 &         0 & 4 \\
        N_1 & M_2, M_4, N_2 &      6 &            1 &         0 & 3 \\
        M_2 & M_4, N_2      &      6 &            2 &         0 & 2 \\
        M_4 & M_2, N_2      &      6 &            3 &         0 & 2 \\
        M_2 \cap M_4 &      &      6 &            6 &         0 & 1 \\
        N_2 & M_2, M_4      &      6 &            1 &         1 & 1 \\
        M_2 \cap N_2 &      &      6 &            2 &         2 & 0 \\
        M_4 \cap N_2 &      &      6 &            3 &         3 & 0 \\\hline
     \end{array}
  \]
  \end{minipage}
  \hfill
  \begin{minipage}{0.45\textwidth}
    \[ \xymatrix@R=5mm@C=8mm%
                {& M \ar@{-}[d] \\
                 & N_1 \ar@{-}[dr] \ar@{-}[d] \ar@{-}[ddl] \\
                 & M_2 \ar@{-}[ddl] \ar@{-}[d] & M_4 \ar@{-}[dl] \ar@{-}[ddl] \\
                 N_2 \ar@{-}[d] \ar@{-}[dr] & M_2 \cap M_4 & \\
                 M_2 \cap N_2 & M_4 \cap N_2
                }
    \]
  \end{minipage}
  \vspace{3mm}
  \caption{Curves of geometric genus~$1$ on~$S$, see Theorem~\ref{T:lowgenus}.
           We use the notations $M_j$, $N_j$ as introduced in
           Sections~\ref{S:moduli} and~\ref{S:low}. Each line applies to surfaces
           whose moduli point is in the subset of~$M$ given in the `$\in$' column,
           but not in any of the subsets given in the `$\notin$' column. \newline
           $M_2 \cap N_2$ is the single point $(u_1,u_2,u_3,d,v,w) = (36,0,0,0,4,0)$
           represented by $(a_1,a_2,a_3,c) = (0,0,6,1)$, whereas $M_4 \cap N_2$
           consists of the two points $(3\alpha^2,3\alpha^4,\alpha^6,0,4,2\alpha^3)$,
           where $4\alpha^2 - 15\alpha + 18 = 0$, represented by $(-\alpha,-\alpha,-\alpha,i)$.}
  \label{Tb:lowg}
  \vspace*{-5mm}\hrulefill
\end{table}

We summarize our findings.

\begin{Theorem} \label{T:lowgenus}
  Let $S$ be a primary Burniat surface. Then $S$ does not contain rational curves.
  Depending on the location of the moduli point of~$S$ in~$M$,
  $S$ contains the numbers of curves of geometric genus~1
  at infinity, of type~I and of type~II as given in Table~\ref{Tb:lowg}.
  The column labeled `dim.'
  gives the dimension of the corresponding subset of the moduli space.

  In particular, the number of curves of geometric genus~1 on a primary
  Burniat surface can be 6, 7, 8, 9, 10 or~12.

  In each case, $\Aut(S)$ acts transitively on the set of curves of type~I
  and on the set of curves of type~II.
\end{Theorem}

\begin{Definition}
  For our purposes, we will say that a primary Burniat surface~$S$ is \emph{generic}, if
  $S$ does not contain curves of type~I or~II.
\end{Definition}

By Theorem~\ref{T:lowgenus}, $S$ is generic if and only if no two of the~$E_j$
are isomorphic as double covers of~$\P^1$, which means that $d \neq 0$.

\begin{figure}[htb]
  \hrulefill
  \[ \xymatrix@C=22mm@R=15mm@!0%
              { 4 & & & M \ar@{-}[lldd] \ar@{-}[rrd] \ar@{-}[d] \\
                3 & & & N_1 \ar@{-}[rrdd] \ar@{-}[d] \ar@{-}[rrrd] \ar@{-}[rrd] \ar@{-}[lldd]
                  & & M_1 \ar@{-}[lldd] \ar@{-}[d]
                  & \\
                2 & M_3 \ar@{-}[d] \ar@{-}[rd]
                  & & M_2 \ar@{-}[ld] \ar@{-}[d] \ar@{-}[rd]
                  & & M_1 \cap N_1 \ar@{-}[d] \ar@{-}[lld] \ar@{-}[rd]
                  & M_4 \ar@{-}[lld] \ar@{-}[d] \\
                1 & (M_3 \cap N_1)'
                  & M_2 \cap M_3
                  & M_1 \cap M_2 \ar@{-}[rd] \ar@{-}[rrd]
                  & M_2 \cap M_4 \ar@{-}[rd]
                  & N_2 \ar@{-}[rd] \ar@{-}[ld]
                  & M_1 \cap M_4 \ar@{-}[ld] \ar@{-}[d] \\
                0 & & & & M_2 \cap N_2
                  & M_1 \cap M_2 \cap M_4
                  & M_4 \cap N_2
              }
  \]
  \caption{The combined stratification. Note that $M_3 \cap N_1$ is the disjoint union
           of $M_2 \cap M_3$ and $(M_3 \cap N_1)'$. The latter is represented by
           $a_1 = 0$, $a_2 = a_3$, $c^4 + 1 = 0$.}
  \hrulefill
\end{figure}
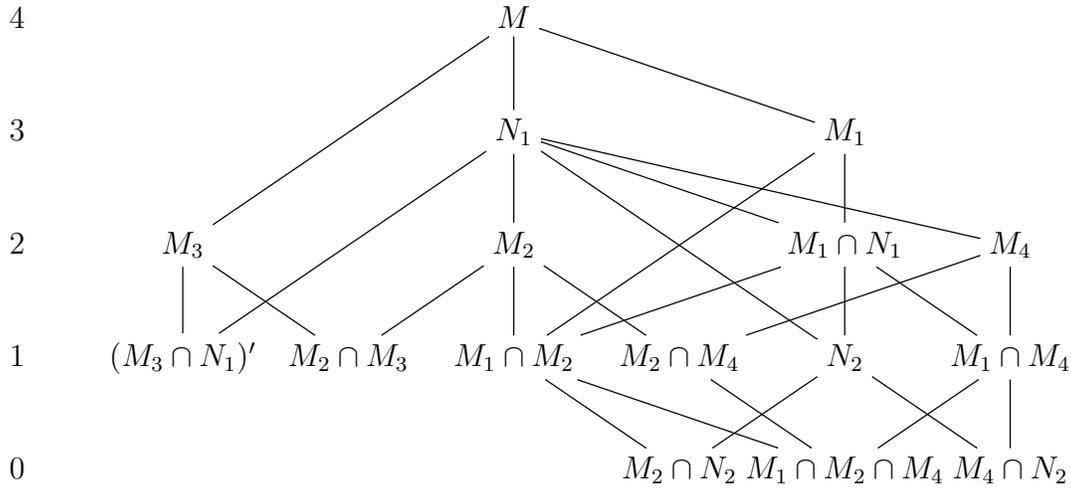


\section{Rational points on primary Burniat surfaces} \label{S:ratpts} \label{S:explicit}

In this section we will assume that everything is defined over~$\Q$.
Concretely, this
means that the curves $E_j$ are curves of genus~$1$ over~$\Q$,
the torsion points~$T_j$ are rational points on the Jacobian elliptic
curves of the~$E_j$, the functions~$x_j$ are defined over~$\Q$, and $c \in \Q$.
Then $X$ and~$S$ are also defined over~$\Q$, and it makes sense
to study the set~$S(\Q)$ of rational points on~$S$.

We recall that $S$ contains six genus~$1$ curves at infinity;
we will soon see that these curves
always have rational points, so they are elliptic curves over~$\Q$.
If some of these curves have positive Mordell-Weil rank, then
$S(\Q)$ is infinite. The same is true when $S$ contains curves
of type~I or~II that are defined over~$\Q$, have a rational point
and positive Mordell-Weil rank.

\begin{Definition}
  Let $S$ be a primary Burniat surface. We denote by $S'$ the complement
  in~$S$ of the union of the curves of geometric genus~$1$ on~$S$.
  Note that $S'(\Q)$ is the set of sporadic rational points on~$S$.
\end{Definition}

\begin{Theorem} \label{T:finite}
  If $S$ a primary Burniat surface, then the set $S'(\Q)$ of sporadic
  rational points on~$S$ is finite.
\end{Theorem}

\begin{proof}
  This is a special case of Theorem~\ref{T:general}. We use the fact
  that the only low-genus curves on~$S$ are the curves of geometric genus~$1$
  that we have excluded from~$S'$.
\end{proof}

In the remainder of this paper, we will be concerned with computing
the set of sporadic rational points explicitly.

We now discuss how to modify the description given in Section~\ref{S:moduli}
when we want to work over~$\Q$ instead of over~$\C$. We will restrict
to the case that the three genus~$1$ curves $E_1$, $E_2$, $E_3$ whose
product contains~$X$ are individually defined over~$\Q$.
Then the two commuting involutions on~$E_j$ used in the construction, one without fixed points
and one with four fixed points, are also defined over~$\Q$.
Dividing by the latter, we see that $E_j$
is (as before) a double cover of~$\BP^1$ ramified in four points, so it
can be given by an equation of the form $y^2 = f(x)$ with a polynomial~$f$
of degree (at most)~4. The function that is invariant under the second
involution and changes sign under the first must then be a function
on~$\BP^1_x$; we can choose $x$
to be that function. Then the first involution is $(x,y) \mapsto (-x,-y)$
(without the sign change on~$y$, it would have fixed points), and
the polynomial~$f$ must be even. So the curve $E_j$ has (affine) equation
\[ E_j \colon y_j^2 = r_j x_j^4 + s_j x_j^2 + t_j \]
with $r_j, s_j, t_j \in \Q$ and $(s_j^2 - 4 r_j t_j) r_j t_j \neq 0$.
Note that the parameter~$a_j$ in equations~\eqref{E:eqns} is then given
by $a_j^2 = s_j^2/r_j t_j$.
The surface $X$ is given by $x_1 x_2 x_3 = c$ inside $E_1 \times E_2 \times E_3$.
The moduli point on~$M$ is given by
\begin{gather*}
  u_1 = \frac{s_1^2}{r_1 t_1} + \frac{s_2^2}{r_2 t_2} + \frac{s_3^2}{r_3 t_3}, \;
  u_2 = \frac{s_1^2 s_2^2}{r_1 r_2 t_1 t_2} + \frac{s_2^2 s_3^2}{r_2 r_3 t_2 t_3}
         + \frac{s_3^2 s_1^2}{r_3 r_1 t_3 t_1}, \;
  u_3 = \frac{s_1^2 s_2^2 s_3^2}{r_1 r_2 r_3 t_1 t_2 t_3}, \\[2mm]
    d = \left(\frac{s_1^2}{r_1 t_1} - \frac{s_2^2}{r_2 t_2}\right)
        \left(\frac{s_2^2}{r_2 t_2} - \frac{s_3^2}{r_3 t_3}\right)
        \left(\frac{s_3^2}{r_3 t_3} - \frac{s_1^2}{r_1 t_1}\right), \;
    v = \frac{r_1 r_2 r_3}{t_1 t_2 t_3} c^4 + 2 + \frac{t_1 t_2 t_3}{r_1 r_2 r_3} c^{-4} \\[2mm]
    \quad\text{and}\quad
    w = s_1 s_2 s_3 \left(\frac{c^2}{t_1 t_2 t_3} + \frac{c^{-2}}{r_1 r_2 r_3}\right)\,.
\end{gather*}
The generators of~$\Gamma$ act as follows.
\begin{align*}
  \gamma_1 &\colon (x_1, y_1, x_2, y_2, x_3, y_3)
                \longmapsto (x_1, y_1, -x_2, -y_2, -x_3, y_3) \\
  \gamma_2 &\colon (x_1, y_1, x_2, y_2, x_3, y_3)
                \longmapsto (-x_1, y_1, x_2, y_2, -x_3, -y_3) \\
  \gamma_3 &\colon (x_1, y_1, x_2, y_2, x_3, y_3)
                \longmapsto (-x_1, -y_1, -x_2, y_2, x_3, y_3) \,.
\end{align*}

\begin{Remark} \label{R:cyclic}
  One also obtains a primary Burniat surface defined over~$\Q$, when the
  three curves~$E_j$ are defined over a cyclic cubic extension~$K$ of~$\Q$
  and form a Galois orbit (and $c \in \Q$).
  Then $\Gamma$ is a $\Q$-group scheme that splits over~$K$.
  In Burniat's construction, this means that the three points~$p_i$
  are defined over~$K$ and permuted cyclically by the action of the
  Galois group of $K$ over~$\Q$.
  We do not look into this case further, or consider the question for which
  moduli points $p \in M(\Q)$ a corresponding primary Burniat surface can
  be defined over~$\Q$.
\end{Remark}

By Remark~\ref{R:twists}, every rational point on~$S$ lifts to one
of finitely many twists $X_\xi \to S$ of $X \to S$. Our method for
determining~$S(\Q)$ will be to first find the relevant twists~$X_\xi$
and then determine the set of rational points~$X_\xi(\Q)$ on each of them.
The first step in this approach is to give an explicit description
of the twists of $X \to S$. In general, these twists are parameterized
by the Galois cohomology set $H^1(\Q, \Gamma)$; in our case, we obtain
the Galois cohomology group (it is a group, because $\Gamma$ is abelian)
\[ H^1(\Q, \Gamma) \cong H^1(\Q, \mu_2)^3
                   \cong \bigl(\Q^\times/(\Q^\times)^2\bigr)^3 \,.
\]
Its elements can be represented by triples $(d_1, d_2, d_3)$ of squarefree
integers. The corresponding twist $X_\xi$ is then given by
\begin{equation} \label{E:twist} \renewcommand{\arraystretch}{1.2}
  X_{(d_1,d_2,d_3)} \colon \left\{%
  \begin{array}{r@{{}={}}l}
    d_3 y_1^2 & r_1 d_2^2 d_3^2 x_1^4 + s_1 d_2 d_3 x_1^2 + t_1 \\
    d_1 y_2^2 & r_2 d_1^2 d_3^2 x_2^4 + s_2 d_1 d_3 x_2^2 + t_2 \\
    d_2 y_3^2 & r_3 d_1^2 d_2^2 x_3^4 + s_3 d_1 d_2 x_3^2 + t_3 \\
    d_1 d_2 d_3 x_1 x_2 x_3 & c \,.
  \end{array}\right.
\end{equation}

We can restrict the prime divisors of the~$d_j$ to the set consisting
of the prime~2 and the primes~$p$ of bad reduction, i.e., such that
the reduction of $X$ contains a fixed point of $\gamma = \gamma_1 \gamma_2 \gamma_3$.
To make this explicit,
we have to define a kind of `discriminant' for~$X$ that vanishes if
and only if $X$ passes through one of the fixed points. This will be the
case if and only if $c = x_1(P_1) x_2(P_2) x_3(P_3)$ for some $P_j \in E_j$
that are fixed points of $(x_j,y_j) \mapsto (x_j,-y_j)$, i.e., for
points $P_j = (\alpha_j, 0)$ with $r_j \alpha_j^4 + s_j \alpha_j^2 + t_j = 0$.
Equivalently, $c^2 = \beta_1 \beta_2 \beta_3$ with $\beta_j$ a root
of~$r_j X^2 + s_j X + t_j$. We therefore look at
\[   \prod_{\eps_1,\eps_2,\eps_3 = \pm 1}
          \bigl(8 r_1 r_2 r_3 c^2
                   - (-s_1 + \eps_1 \Delta_1)
                     (-s_2 + \eps_2 \Delta_2)
                     (-s_3 + \eps_3 \Delta_3)\bigr) \,,
\]
where $\Delta_j = \sqrt{s_j^2 - 4 r_j t_j}$. We can divide this by
$2^{24} r_1^4 r_2^4 r_3^4$ and obtain
\begin{align*}
  D &= r_1^4 r_2^4 r_3^4 \, c^{16} + r_1^3 r_2^3 r_3^3 s_1 s_2 s_3 \, c^{14}
          + r_1^2 r_2^2 r_3^2 (4 \sigma_1 - 2 \sigma_2 + \sigma_3) \, c^{12} \\
    & \quad{} + r_1 r_2 r_3 s_1 s_2 s_3 (-5 \sigma_1 + \sigma_2) \, c^{10}
              + (6 \sigma_1^2 - 4 \sigma_1 \sigma_2 - 2 \sigma_1 \sigma_3
                  + \sigma_1 \sigma_4 + \sigma_2^2) \, c^8 \\
    & \quad{} + s_1 s_2 s_3 t_1 t_2 t_3 (-5 \sigma_1 + \sigma_2) \, c^6
              + t_1^2 t_2^2 t_3^2 (4 \sigma_1 - 2 \sigma_2 + \sigma_3) \, c^4 \\
    & \quad{} + s_1 s_2 s_3 t_1^3 t_2^3 t_3^3 \, c^2  + t_1^4 t_2^4 t_3^4 \,.
\end{align*}
Here we have set
\begin{align*}
  \sigma_1 &= r_1 r_2 r_3 t_1 t_2 t_3 \\
  \sigma_2 &= r_1 r_2 s_3^2 t_1 t_2 + r_1 r_3 s_2^2 t_1 t_3 + r_2 r_3 s_1^2 t_2 t_3 \\
  \sigma_3 &= r_1 s_2^2 s_3^2 t_1 + r_2 s_1^2 s_3^2 t_2 + r_3 s_1^2 s_2^2 t_3 \\
  \sigma_4 &= s_1^2 s_2^2 s_3^2 \,.
\end{align*}
We see that $X$ does not pass through fixed points of~$\gamma$ and is smooth
(and therefore $S$ is smooth) if and only if $c D \neq 0$ and the curves $E_j$
are smooth. The latter condition is
\[ 0 \neq \prod_{j=1}^3 (r_j t_j (s_j^2 - 4 r_j t_j))
      = \sigma_1 (\sigma_4 - 4 \sigma_3 + 16 \sigma_2 - 64 \sigma_1) \,.
\]

We obtain the following result.

\begin{Lemma}
  Assume that the $r_j, s_j, t_j$ and $c$ are integers satisfying
  \[ r_j t_j (s_j^2 - 4 r_j t_j) \neq 0 \quad\text{for all~$j$ \qquad and}
      \qquad c \neq 0, \quad D \neq 0.
  \]
  Let $d_1, d_2, d_3$ be squarefree integers.
  Then the twist $X_{(d_1,d_2,d_3)}$ of $X \to S$ fails to have
  points over some completion of~$\Q$ unless all the prime divisors
  of $d_1 d_2 d_3$ are prime divisors
  of~$2 c D \sigma_1 (\sigma_4 - 4 \sigma_3 + 16 \sigma_2 - 64 \sigma_1)$.
\end{Lemma}

We now consider the genus~1 curves at infinity on~$X$ and on~$S$. On~$X$, a typical
$\Gamma$-orbit of these is as follows. We fix a point at~infinity on~$E_1$
and a point with zero $x$-coordinate on~$E_2$, say those with
$y_1/x_1^2 = \sqrt{r_1}$ and $y_2 = \sqrt{t_2}$; call them $P$ and~$Q$
and the other points with the same $x$-coordinate $\bar{P}$ and~$\bar{Q}$.
Then the curve $\{P\} \times \{Q\} \times E_3$ is contained in~$X$, and
the generators of~$\Gamma$ act as follows.
\begin{align*}
  \gamma_1 &\colon (P, Q, R) \longmapsto (P, \bar{Q}, -R + T_3) \\
  \gamma_2 &\colon (P, Q, R) \longmapsto (P, Q, R + T_3) \\
  \gamma_3 &\colon (P, Q, R) \longmapsto (\bar{P}, Q, R)\,.
\end{align*}
The orbit of the curve consists of four curves, the other three being
$\{\bar{P}\} \times \{Q\} \times E_3$, $\{P\} \times \{\bar{Q}\} \times E_3$
and $\{\bar{P}\} \times \{\bar{Q}\} \times E_3$. The action splits into
independent actions on $\{P, \bar{P}\}$ and on $\{Q, \bar{Q}\} \times E_3$.
Invariants for the latter are given by $x_3^2$ and~$\sqrt{t_2} x_3 y_3$
(where $Q \mapsto \bar{Q}$ corresponds to $\sqrt{t_2} \mapsto -\sqrt{t_2}$).
There is the relation
\[ (\sqrt{t_2} x_3 y_3)^2 = t_2 x_3^2 (r_3 x_3^4 + s_3 x_3^2 + t_3) \,, \]
which in terms of $Y = \sqrt{t_2} x_3 y_3$ and~$X = x_3^2$ can be written as
\[ Y^2 = t_2 X (r_3 X^2 + s_3 X + t_3) \,, \]
which is the quadratic twist by~$t_2$ of the Jacobian elliptic curve of~$E_3$.
Switching the roles of zero and infinity, we get the $r_2$-twist of the
same elliptic curve. We obtain the other curves by cyclic permutation of
the indices. This gives us the quadratic twists by~$t_3$ and by~$r_3$ of
the Jacobian elliptic curve of~$E_1$ and the quadratic twists by~$t_1$ and
by~$r_1$ of the Jacobian elliptic curve of~$E_2$. These six curves are
arranged in the form of a hexagon, with intersection points at the origin
and the obvious (rational) point of order~2. We can now give the
following refinement of Theorem~\ref{T:finite}:

\begin{Theorem}
  Assume that $E_1$, $E_2$, $E_3$ are defined over~$\Q$ and that
  $T_1$, $T_2$, $T_3$ are rational points. Then we
  always have $\#S(\Q) \ge 6$. $S'(\Q)$ is finite,
  and $S(\Q)$ is infinite if and only if at least one of the genus~$1$ curves
  on~$S$ has a rational point and positive rank.
\end{Theorem}

\begin{proof}
  The first statement follows, since the six points where
  adjacent elliptic curves at infinity in~$S$ intersect
  (compare Figure~\ref{Fighex}) are always rational. For the
  second statement, we recall that $S'(\Q)$ is finite by Theorem~\ref{T:finite}.
  Therefore $S(\Q)$ is infinite if and only if one of the genus~$1$ curves
  on~$S$ has infinitely many points, i.e., it has a rational point and positive rank.
\end{proof}

\begin{Remark}
  If we consider a primary Burniat surface defined over~$\Q$ as described in
  Remark~\ref{R:cyclic}, then the absolute Galois group of~$\Q$ acts on the
  six curves at infinity and therefore also on their intersection points
  by cyclic permutations of order~$3$. So in this case, these six intersection
  points will not be rational (they are defined over~$K$).
\end{Remark}

\begin{Remark}
  The preceding theorem shows that $S$ always has rational points
  (under the assumptions of the theorem). This is not true in general for its
  \'etale covering~$X$. For example, we can take $E_1$ to be without rational
  points as for the surface considered in Proposition~\ref{P:Example1} below.
  However, there will always be a twist $X' \to S$ of the covering
  such that $X'$ has rational points.
\end{Remark}


\section{Determining the set of rational points on the twists} \label{S:find}

This section is dedicated to the discussion how to  determine the set $X_\xi(\Q)$ of rational points
on  the relevant twists of~$X$. For this, note first that $X_\xi$
is of the same form as~$X$, as shown by the defining equations~\eqref{E:twist}.
For the description of the available methods, it is therefore sufficient
to deal with~$X$ itself.

We have the three projections $\psi_j \colon X \inj E_1 \times E_2 \times E_3 \to E_j$
at our disposal. There are two cases.

\subsection{Some $E_j$ fails to have rational points}

If we can show that $E_j(\Q) = \emptyset$ for some~$j$,
then clearly $X(\Q) = \emptyset$, too. To decide whether $E_j$ has rational
points, we can, on the one hand, try to find a rational point by a systematic
search, and, on the other hand, try to prove that no rational point exists
by performing a `second descent' on the 2-covering $E_j$ of~$J_j$, its Jacobian
elliptic curve. The latter can be done explicitly by computing the
`2-Selmer set' of~$E_j$, compare~\cite{BruinStoll2009}; if this set is empty,
then $E_j(\Q) = \emptyset$ as well. If necessary, a further descent step
can be performed on the two-coverings of~$E_j$, compare~\cite{Stamminger}.

\subsection{All $E_j$ have rational points}

If we find some $P_j \in E_j(\Q)$ for all~$j$, then the~$E_j$ are elliptic curves,
and we can try to determine the free abelian rank of the groups~$E_j(\Q)$.
The standard tool for this is Cremona's {\sf mwrank} program as described
in~\cite{CremonaBook}, whose functionality is included in Magma and SAGE.
This is based on 2-descent. If this is not sufficient to determine the
rank, then one can also perform `higher descents', see for example~\cite{CFOSS}
and the references given there.

There are several cases, according to how many of the sets~$E_j(\Q)$ are finite.

\subsection*{All $E_j(\Q)$ are finite}

Then $(E_1 \times E_2 \times E_3)(\Q)$ is finite, and we only have to check which
of these finitely many points lie on~$X$.

\subsection*{Two of the $E_j(\Q)$ are finite}

Say $E_1(\Q)$ and~$E_2(\Q)$ are finite. Then
we can easily determine the set of rational points on~$X$ by checking the
finitely many fibers of $(\psi_1, \psi_2) \colon X \to E_1 \times E_2$
above rational points of $E_1 \times E_2$; note that these fibers are finite
(of size two), hence it is easy to determine their rational points.

\subsection*{Exactly one of the $E_j(\Q)$ is finite}

Say $E_1(\Q)$ is finite. In this case, we can at least reduce to the
finitely many fibers of $\psi_1$ over rational points on~$E_1$. These fibers
are (generically) smooth curves of genus~5 contained in the product of the
two other genus~1 curves (and defined by the relation that the product
of the two $x$-coordinates is constant).

We then need a way of determining
the set of rational points on  a genus $5$ curve~$C$ sitting in a product of two elliptic
curves of positive rank. After making a transformation $x \leftarrow c'/x$
of the $x$-coordinate on one of the two genus~$1$ curves, the product condition
becomes the condition that the two $x$-coordinates are equal.
The curve~$C$ can then be written as a bidouble cover of~$\BP^1$,
given by a pair of equations
\[ u^2 = a_0 x^4 + a_1 x^2 z^2 + a_2 z^4\,, \qquad
   v^2 = b_0 x^4 + b_1 x^2 z^2 + b_2 z^4\,.
\]
This curve is a double cover of the hyperelliptic curve
\[ D \colon y^2 = (a_0 x^4 + a_1 x^2 z^2 + a_2 z^4) (b_0 x^4 + b_1 x^2 z^2 + b_2 z^4) \,, \]
which is of genus~3 (unless the two factors have common roots), so the Jacobian
of~$C$ splits up to isogeny into a product of three elliptic curves
(with the Jacobians of $E_1$ and~$E_2$ among them) and an \Abelian\  surface.
More precisely, $D$ covers the genus~1 curve
\[ E \colon y^2 = (a_0 {x'}^2 + a_1 x' z' + a_2 {z'}^2) (b_0 {x'}^2 + b_1 x' z' + b_2 {z'}^2) \]
and the genus~2 curve
\[ F \colon w^2 = x' z' (a_0 {x'}^2 + a_1 x' z' + a_2 {z'}^2) (b_0 {x'}^2 + b_1 x' z' + b_2 {z'}^2) \]
(where $(x' : z' : w) = (x^2 : z^2 : xyz)$).
 If the third elliptic curve has rank zero or the Jacobian of~$F$ has Mordell-Weil
rank at most~1 (the rank can usually be determined by the methods of~\cite{Stoll2Desc}),
then methods are available that in many cases will be able
to determine~$C(\Q)$, see for example~\cite{BruinStollMWS}. Even when the
rank is 2 or larger, `Elliptic curve Chabauty' methods might apply to $F$ or also~$D$,
see~\cite{BruinECC}.

The possible degenerations of a fiber of the genus $5$ fibration $X \rightarrow E_j$ are as follows.
\begin{enumerate}[(1)]
  \item A curve of geometric genus~$3$ with two nodes. Then $D$ and~$F$ above
        each degenerate to a curve of (geometric) genus~$1$, and we can hope
        that at least one of them has only finitely many rational points.
  \item Two curves of geometric genus~$1$ intersecting in four points.
        These map to curves of type~I on~$S$; we can try to find out whether
        they have finitely many or infinitely many rational points.
 \end{enumerate}

\subsection*{All $E_j(\Q)$ are infinite}

In this situation, there must be finitely many points $p_1, \ldots , p_r \in E_j(\Q)$, such that the sporadic rational points of $X$ are contained in the fibers over $p_1, \ldots , p_r$. But there are no methods available so far to determine these points, hence we cannot yet deal with this situation.

We end this section by taking a quick look at the situation when there are curves of type I or~II
on~$S$ and observe the following.
\begin{Remark}
  If a curve of type I or~II on~$S$ has only finitely many rational points, it will not present
  any difficulties. Otherwise:
  \begin{enumerate}[(I)]
    \item When there is a curve~$E$ of type~I (defined over~$\Q$) on~$S$ such
          that $E(\Q)$ is infinite, then there will be some twist~$X_\xi$ of~$X$
          such that $E$ lifts to an elliptic curve~$E'$ on~$X_\xi$
          with infinitely many rational points. Since $E'$ is isomorphic to two
          of the~$E_{j,\xi}$, we need the third~$E_{j,\xi}$ to be of rank~$0$
          for our approach to work.
    \item When there is a curve~$E$ of type~II (defined over~$\Q$) on~$S$ such
          that $E(\Q)$ is infinite, then there will again be some twist~$X_\xi$ of~$X$
          such that $E$ lifts to an elliptic curve~$E'$ on~$X_\xi$
          with infinitely many rational points. In this case, $E'$ is isogenous
          to all three~$E_{j,\xi}$, so all three curves will have positive rank,
          and our current methods will necessarily fail.
  \end{enumerate}
\end{Remark}


\section{Examples} \label{S:examples}

We first give a number of examples in the generic case, when $S$ does
not have curves of types I or~II.

As a first example, take
\begin{equation} \label{E:ex1}
  E_1 \colon y_1^2 = -x_1^4 - 1\,, \quad
  E_2 \colon y_2^2 = x_2^4 - 1\,, \quad
  E_3 \colon y_3^2 = x_3^4 + 1
\end{equation}
and $c = 2$. All three curves have bad reduction only at~$2$, and
$D = 50625 = 3^4 \cdot 5^4$, so the set of bad primes is $\{2,3,5\}$.
There are $2^{3 \cdot 4} = 4096$ triples $(d_1, d_2, d_3)$ to be checked.
Of these, only $(1,-1,1)$, $(1,-1,-1)$, $(2,-2,2)$ and $(2,-2,-2)$
satisfy the condition that the twisted product of the~$E_j$ has a rational
point. It turns out that all the twisted genus~1 curves that occur have
only finitely many points. We can find them all and thence determine
the sets $X_{(d_1,d_2,d_3)}(\Q)$. It turns out that these sets are empty
for the triples $(2,-2,2)$ and $(2,-2,-2)$, and for the other two triples
all these points map to points at infinity on~$S$; they fall into five
orbits each under the action of~$\Gamma$ on the two twists.
This results in the following.

\begin{Proposition} \label{P:Example1}
  Let $X \subset E_1 \times E_2 \times E_3$ with $E_j$ as in~\eqref{E:ex1} be defined
  by $x_1 x_2 x_3 = 2$ and set $S = X/\Gamma$. Then $S'(\Q) = \emptyset$ and
  $\#S(\Q) = 10$.
\end{Proposition}

As our second example, we consider
\begin{equation} \label{E:ex2}
  E_1 \colon y_1^2 = -x_1^4 - 1\,, \quad
  E_2 \colon y_2^2 = x_2^4 + 2 x_2^2 + 2\,, \quad
  E_3 \colon y_3^2 = -2 x_3^4 - 2 x_3^2 + 1
\end{equation}
and $c = 1$. The curves have bad reduction at most at $2$ and~$3$, and
$D = 256 = 2^8$, so the set of bad primes is~$\{2,3\}$. There are six triples
$(d_1,d_2,d_3)$ such that the corresponding twist of $E_1 \times E_2 \times E_3$
has rational points, namely $(1, -1, 1)$, $(1, -1, 2)$, $(1, -2, 2)$, $(-2, -1, 1)$,
$(-2, -1, 2)$ and $(-2, -2, 2)$. For the third and the last, the twisted product
has only finitely many rational points. The third gives no contribution (the rational
points are not on the twist of~$X$), whereas the last gives two rational
points at infinity on~$S$. For the remaining four triples, exactly one of the
twisted curves has infinitely many rational points (rank~$1$).
However, the other two curves have only rational points with $x_j \in  \{0, \infty\}$,
so all rational points on the corresponding twists of~$X$ map to points at infinity
of~$S$, of which there are now infinitely many (two of the six genus~1 curves
at infinity on~$S$ have infinitely many rational points). We obtain the following.

\begin{Proposition}
  Let $X \subset E_1 \times E_2 \times E_3$ with $E_j$ as in~\eqref{E:ex2} be defined
  by $x_1 x_2 x_3 = 1$ and set $S = X/\Gamma$. Then $S'(\Q) = \emptyset$ and
  $S(\Q)$ is infinite.
\end{Proposition}

As our third example, we take
\begin{equation} \label{E:ex3}
  E_1 \colon y_1^2 = 2 x_1^4 + x_1^2 + 1\,, \quad
  E_2 \colon y_2^2 = x_2^4 - x_2^2 + 1\,, \quad
  E_3 \colon y_3^2 = x_3^4 - x_3^2 + 4
\end{equation}
and $c = 1$. The curves have bad reduction at $2$, $3$, $5$ and~$7$, and
$D = 436 = 2^2 \cdot 109$, so that the set of bad primes is~$\{2, 3, 5, 7, 109\}$.
There are two triples $(d_1,d_2,d_3)$ such that the corresponding twist of the product
has rational points; they are $(1, 1, 1)$ and~$(1, 1, 2)$. In both cases, two of
the twisted curves have rank~$1$ and the third has finitely many rational points.
For the second triple, these points all have $x \in \{0, \infty\}$, so that we
do not obtain a contribution to~$S'(\Q)$. For the first triple, the twist of~$E_2$
has eight rational points, four of which have $x \in \{0, \infty\}$ and the other
four have $x \in \{\pm 1\}$. For each of these four points, the curve $D$ is given by
\[ D \colon y^2 = (2 x^4 + x^2 + 1)(4 x^4 - x^2 + 1) \,. \]
The elliptic curve~$E$ covered by it has rank~$1$, but for the genus~2 curve
\[ F \colon y^2 = x (2 x^2 + x + 1) (4 x^2 - x + 1) \]
that is also covered by~$D$, one can show (using the algorithm of~\cite{Stoll2Desc})
that its Jacobian has Mordell-Weil rank~$1$. Then Chabauty's method combined with
the Mordell-Weil sieve (see~\cite{BruinStollMWS}) can be used to show that
\[ F(\Q) = \{\infty, (0, 0), (1, 4), (1, -4)\} \,. \]
Note that the twist $(1,1,1)$ is the original surface~$X$. Translating back to the
coordinates on $E_1$, $E_2$, $E_3$, we see that all points on~$X$ must have
$x_1, x_2, x_3 \in \{\pm 1\}$. Since each of the three curves has a pair of rational
points with $x = 1$ and also with $x = -1$, this gives a total of $32$~rational
points on~$X$ mapping to~$S'$. They fall into four $\Gamma$-orbits and therefore
give rise to four rational points on~$S'$. Since of the six elliptic curves at infinity
on~$S$, four have positive rank, we obtain the following.

\begin{Proposition}
  Let $X \subset E_1 \times E_2 \times E_3$ with $E_j$ as in~\eqref{E:ex3} be defined
  by $x_1 x_2 x_3 = 1$ and set $S = X/\Gamma$. Then $S'(\Q)$ consists of four points and
  $S(\Q)$ is infinite.
\end{Proposition}

To conclude, we present an example where $S$ contains a curve of type~I with infinitely
many rational points. We take
\begin{equation} \label{E:ex4}
  E_1 \colon y_1^2 = 2 x_1^4 + x_1^2 + 1\,, \quad
  E_2 \colon y_2^2 = x_1^4 + x_2^2 + 2\,, \quad
  E_3 \colon y_3^2 = x_3^4 - x_3^2 + 1
\end{equation}
and $c = 1$. The curves have bad reduction at $2$, $3$ and~$7$, and
$D = 100 = 2^2 \cdot 5^2$, so that the set of bad primes is~$\{2, 3, 5, 7\}$.
There are four triples $(d_1,d_2,d_3)$ such that the corresponding twist of the product
has rational points; they are $(1, 1, 1)$, $(1, 1, 2)$, $(2, 1, 1)$ and~$(2, 1, 2)$.
For the twists corresponding to the last three, there is always one twisted~$E_j$
that has only finitely many rational points, all of which have $x \in \{0, \infty\}$,
so we obtain no contribution outside the curves at infinity on~$S$ from these twists.
For the first twist (which is $X$ itself), we have that $E_1$ and~$E_2$ both have rank~$1$
and $E_3$ has eight rational points, four of which have $x \notin \{0, \infty\}$.
The fibers above these four points form a $\Gamma$-orbit, and each fiber splits
as a union of two elliptic curves intersecting in four points, thus producing
a type~I curve on~$S$. The curves on~$X$ are isomorphic to~$E_1$ (and~$E_2$); they
therefore have infinitely many rational points, and the same is true for the
type~I curve on~$S$ they map to. Since there are no further contributions, we
have the following result.

\begin{Proposition}
  Let $X \subset E_1 \times E_2 \times E_3$ with $E_j$ as in~\eqref{E:ex4} be defined
  by $x_1 x_2 x_3 = 1$ and set $S = X/\Gamma$. Then $S$ contains exactly one curve
  of type~I, which has infinitely many rational points, and no curve of type~II.
  Furthermore, three out of the six
  curves at infinity on~$S$ contain infinitely many rational points. There are
  no sporadic rational points on~$S$.
\end{Proposition}

\begin{proof}
  It remains to show that the type~I curve detected above is the only extra low-genus
  curve on~$S$.
  This follows from the discussion in Section~\ref{S:low}, since the moduli point
  of~$S$ is given by $(u_1,u_2,u_3,d,v,w) = (2,\frac{5}{4},\frac{1}{4},0,4,-1)$,
  which is in~$N_1$, but neither in $M_2$, $M_4$ nor in~$N_2$.
  Also, one checks that three of the six curves at infinity have positive rank.
\end{proof}


\end{document}